\documentclass[12pt, a4paper]{article}
\usepackage{graphicx}
\usepackage{amssymb}
\usepackage{amsthm}

\usepackage{amsmath}
\usepackage{hyphenat}
\usepackage[normalem]{ulem} %strike out text

\usepackage{examplep}

\usepackage{verbatim}

\usepackage{float}
\usepackage{xspace}
\usepackage{color}
\usepackage[hidelinks]{hyperref} %hidelinks to get rid of the square around hyperlink
\hypersetup{
	colorlinks}%,
\usepackage[font=scriptsize]{caption} 
\newcommand{\ii}{\operatorname{i}}
\newcommand{\ee}{\operatorname{e}}
\newcommand{\eps}{\varepsilon}
\newcommand{\dd}{\mathrm{d}}
\newcommand{\Comma}{\: ,}
\newcommand{\period}{\: .}
\newcommand{\semicolon}{\: ;}
\newcommand{\Ai}{\operatorname{Ai}}
\newcommand{\Bi}{\operatorname{Bi}}
\newcommand{\MATLAB}{\textsc{Matlab}\xspace}

\newcommand{\anton}{\textcolor{black}}
\let\oldfrac\frac% Store \frac
\renewcommand{\frac}[2]{%
	\mathchoice
	{\oldfrac{#1}{#2}}% display style
	{\oldfrac{#1}{#2}} % text style
	{\oldfrac{#1}{#2}} % script style
	{#1/#2}% script-script style
}

\newtheorem{Theorem}{Theorem}[section]
\newtheorem{Remark}[Theorem]{Remark}
\newtheorem{Lemma}[Theorem]{Lemma}

\newtheorem{Proposition}[Theorem]{Proposition}
\newtheorem{Definition}[Theorem]{Definition}

\newtheorem{Hypothesis}{Hypothesis}

\newcommand{\footremember}[2]{%
\footnote{#2}
\newcounter{#1}
\setcounter{#1}{\value{footnote}}%
}

\numberwithin{equation}{section}

\newcounter{bla}

\date{\vspace{-2em}}

\begin{document}

%\begin{frontmatter}


\title{\vspace{-2.5cm}WKB-based third order method for the highly oscillatory 1D stationary Schrödinger equation}

%\author{Anton Arnold}%\corref{author}}
%\ead{anton.arnold@tuwien.ac.at}
%\author{Jannis Körner}
%\ead{jannis.koerner@tuwien.ac.at}

 \author{
    Anton Arnold\footremember{label1}{Corresponding author.}\textsuperscript{,}\footremember{alley}{Institute of Analysis and Scientific              Computing, Technische Universität Wien, Wiedner Hauptstr. 8-10, 1040 Wien, Austria,
    \href{mailto:anton.arnold@tuwien.ac.at}{anton.arnold@tuwien.ac.at}} 
    \and
    Jannis Körner\footremember{trailer}{Institute of Analysis and Scientific              Computing, Technische Universität Wien, Wiedner Hauptstr. 8-10, 1040 Wien, Austria,
    \href{mailto:jannis.koerner@tuwien.ac.at}{jannis.koerner@tuwien.ac.at}}
}

\maketitle

%\cortext[author] {Corresponding author.}
%\address{Institute of Analysis and Scientific Computing, Technische Universität Wien, Wiedner Hauptstr. 8-10, 1040 Wien, Austria}
%
\begin{abstract}
    This paper introduces an efficient high-order numerical method for solving the 1D stationary Schrödinger equation in the highly oscillatory regime. Building upon the ideas from \anton{the article [Arnold, Ben Abdallah, Negulescu, SIAM J.\ Numer.\ Anal., 2011],} %\cite{Arnold2011WKBBasedSF}, 
    we first analytically transform the given equation into a smoother (i.e.\ less oscillatory) equation. By developing sufficiently accurate quadratures for several (iterated) oscillatory integrals occurring in the Picard approximation of the solution, we obtain a one-step method that is third order w.r.t.\ the step size. The accuracy and efficiency of the method are illustrated through several numerical examples.
\end{abstract}

%\begin{keyword}
%	keywords here, in the form: keyword \sep keyword
%	keyword1; keyword2; keyword3; etc.
%	keyword1 \sep keyword2 \sep keyword3
%	Schrödinger equation \sep highly oscillatory wave functions \sep higher order WKB approximation \sep optimal truncation \sep spectral methods \sep asymptotic analysis \sep Airy function	

    %\MSC[2020] 34E20 \sep 81Q20 \sep 65L11 \sep 65M70
%\end{keyword}

%\end{frontmatter}
%
\noindent\textbf{Key words:} Schrödinger equation, highly oscillatory wave functions, higher order WKB approximation, initial value problem

\noindent\textbf{AMS subject classifications:} 34E20, 81Q20, 65L11, 65M70

%\end{frontmatter}
%\maketitle


%
\section{Introduction}
This work is concerned with the numerical treatment of the highly oscillatory 1D Schrödinger equation
\begin{align} \label{schroedinger}
	\varepsilon^{2}\varphi^{\prime\prime}(x) + a(x) \varphi(x) = 0 \Comma\quad x\in\mathbb{R} \Comma
\end{align}
where the parameter $0<\varepsilon\ll 1$ is assumed to be very small. Further, we assume the real-valued coefficient function $a$ to be sufficiently smooth and bounded away from zero, i.e. $a(x)\geq a_{0}>0$. The Schrödinger equation (\ref{schroedinger}) plays an important role within the context of quantum mechanical problems, e.g., for the simulation of electron transport in semiconductor devices \cite{Mennemann2013TransientSS,%Negulescu2005PHD
Neg2008,Sun1998ResonantTD}. In these applications, $\varphi$ represents the (possibly complex-valued) Schrödinger wave function, and $a(x):=E-V(x)$ is associated with a prescribed electrostatic potential $V$, where $E\in\mathbb{R}$ denotes the injection energy of an electron of mass $m$. The small parameter $\varepsilon:=\frac{\hbar}{\sqrt{2m}}$ is then proportional to the (reduced) Planck constant $\hbar$. We note that there are numerous additional applications of equation (\ref{schroedinger}), including plasma physics \cite{Courant1958TheoryOT,Lewis1968MOTIONOA} and cosmology \cite{Martin2003WKBAF,Winitzki2005CosmologicalPP}. 

Since the (local) wave length of a solution $\varphi$ to (\ref{schroedinger}) is given by $\lambda(x)=(2\pi\varepsilon)/\sqrt{a(x)}$, the solution exhibits high oscillations for small values of $\varepsilon$, particularly in the semi-classical limit $\varepsilon \to 0$. Therefore, standard methods for solving (\ref{schroedinger}) are typically constrained by choosing very small grid sizes $h=\mathcal{O}(\eps)$ (e.g., see \cite{IHLENBURG19959}) in order to obtain reasonably accurate numerical solutions. However, this leads to high inefficiencies for small values of $\eps$. Hence, there is a keen interest in numerical methods that allow this grid size limitation to be reduced or even eliminated entirely.

In \cite{Jahnke2003NumericalIF,Lorenz2005AdiabaticIF} the authors proposed efficient and \textit{uniformly accurate} (w.r.t.\ $\eps$) schemes, which allow to reduce the grid size limitation to $h=\mathcal{O}(\sqrt{\eps})$ while yielding global errors of order $\mathcal{O}(h^{2})$. The same grid limitation is achieved with the WKB-based (named after the physicists Wentzel, Kramers, Brillouin; cf.\ \cite{Landau1985Q}) second order (w.r.t.\ the step size $h$) one-step scheme from \cite{Arnold2011WKBBasedSF}. This method relies on an analytical pre-processing of the ODE (\ref{schroedinger}), utilizing a priori information of the solution by considering a second order (w.r.t. $\eps$) WKB approximation. \anton{We remark that this approach only uses the first three WKB-terms, but \emph{not} the whole asymptotic WKB-series, which is typically divergent, see e.g.\ \cite{arnold2024optimally}. By eliminating the dominant oscillations, this procedure transforms \eqref{schroedinger} into an equivalent equation that has either a smooth solution or (at least) oscillations with a rapidly decreasing amplitude as $\eps\to0$. As such, it is strategically similar to the more recent \emph{phase function method} of Bremer \cite{Bremer2018}. Such analytical pre-processing enables the computation of a numerical solution on a coarse grid. 
In fact, the \emph{WKB-marching method}} from \cite{Arnold2011WKBBasedSF} is sometimes even \textit{asymptotically correct}, meaning the numerical error goes to zero as $\eps$ approaches zero. This holds true under the condition that the integrals $\int^{x}\sqrt{a(\tau)}\,\dd\tau$ and $\int^{x}a(\tau)^{-1/4}(a(\tau)^{-1/4})''\,\dd\tau$ for the phase of the solution can be computed exactly. The method then yields numerical errors which are $\mathcal{O}(\eps^{3})$ as $\eps\to 0$.

In \cite{Krner2021WKBbasedSW}, an adaptive step size control and a switching mechanism were added on top of the second order scheme (w.r.t.\ the step size) from \cite{Arnold2011WKBBasedSF}. The switching mechanism allowed the algorithm to switch to a standard ODE method (e.g., Runge-Kutta) during the computation and was implemented to avoid technical or efficiency issues in regions where the coefficient function $a(x)$ is very small or even equal to zero. \anton{\cite{Krner2021WKBbasedSW} was inspired by and compared to the earlier work \cite{Agocs2020}, which uses in the WKB-regime a first order method w.r.t.\ $h$.  }

\anton{
The last few years brought about even more refined numerical methods and algorithms for highly oscillatory problems: In \cite{Bremer2023} and \cite{SeBr2024}, Bremer's phase function method was extended to deal with turning points and to inhomogeneous linear ODEs, respectively. In \cite{Agocs2024}, Agocs \& Barnett developed a coupled algorithm with adaptive step sizes for mixed problems, which may include highly oscillatory and slowly changing regions. In the former regime, the method is based on solving a Riccati equation for a smooth phase function using a defect correction procedure. The article \cite{Agocs2024} also includes a detailed numerical comparison with \cite{Krner2021WKBbasedSW} and \cite{Bremer2018}. Further approaches for highly oscillatory ODEs also include a GMRES-based approach \cite{Olver-GMRES2011}. }

\anton{Next we briefly compare the different focuses in the corresponding literature: The articles \cite{Agocs2024, Bremer2018} have quite algorithmic motivation, providing many practical hints and polished open source software packages. By contrast, this work (as well as \cite{Jahnke2003NumericalIF,Lorenz2005AdiabaticIF,Arnold2011WKBBasedSF}) takes a rather numerical analysis point of view and provides a rigorous error analysis, thus giving theoretical guarantees that other methods may not. 
}

\bigskip

The goal of this work is to improve the method from \cite{Arnold2011WKBBasedSF} by extending it to a third order scheme (w.r.t.\ the step size) \anton{ and to compare its performance to the second order WKB-marching method from \cite{Arnold2011WKBBasedSF} as well as to the \emph{adaptive Riccati defect correction method} from \cite{Agocs2024}. The higher-order WKB-marching method will maintain } the same analytical pre-processing of the given ODE (\ref{schroedinger}) as in \cite{Arnold2011WKBBasedSF}. The derivation of the third order scheme can then be realized mainly by two steps: Firstly, by keeping an additional term of the Picard approximation of the solution; and secondly, by a special treatment of the (iterated) oscillatory integrals which occur in each term of this Picard approximation. While the first step is straightforward, the second step is highly technical and consists of rather extensive computations in order to obtain sufficiently accurate quadratures for the oscillatory integrals. Here, the main strategies employed to achieve the desired accuracy are the same as in \cite{Arnold2011WKBBasedSF}.

This paper is organized as follows: In Section~\ref{chap:WKB-trafo} we provide a short review of the WKB-based transformation from \cite{Arnold2011WKBBasedSF}, which allows transforming (\ref{schroedinger}) into a less oscillatory problem. Section~\ref{chap:construction_scheme} details the construction of sufficiently accurate approximations for the (iterated) oscillatory integrals that appear in the Picard approximation of the solution. Section~\ref{chap:scheme} contains the definition as well as the error analysis of the numerical scheme, with the main result of this paper stated in Theorem \ref{thm_main}. In Section~\ref{chap:simulations} we present numerical simulations and illustrate the theoretical results established in this paper. 
\anton{In Section~\ref{sec:comparison} we shall attempt an efficiency comparison of our numerical scheme with the program {\tt riccati} from \cite{Agocs2024,Agocs2023} --- in spite of its conceptual differences. In two examples we shall see that the WKB-marching method has slight advantages w.r.t. accuracy and/or runtime for very small values of $\varepsilon$.} 
We conclude in Section~\ref{chap:conclusion}.

%
%\section{Review of the WKB-based transformation from \cite{Arnold2011WKBBasedSF}}\label{chap:WKB-trafo}
\section{Review of the WKB-based transformation from \texorpdfstring{\cite{Arnold2011WKBBasedSF}}{[2]}}\label{chap:WKB-trafo}

In this section we shall review the basics of the second order (w.r.t.\ $\eps$) WKB-based transformation from \cite{Arnold2011WKBBasedSF}, which was used there to transform the Schrödinger equation (\ref{schroedinger}) into a smoother (i.e.\ less oscillatory) equation. Based on this analytical pre-processing of (\ref{schroedinger}), the authors developed a second order (w.r.t.\ the step size $h$) numerical scheme that is uniformly correct in $\eps$ and sometimes even asymptotically correct, i.e.\ the numerical error goes to zero with $\varepsilon\to 0$, while the grid size $h$ remains constant.

Our aim here is to solve the following initial value problem (IVP) for the Schrödinger equation (\ref{schroedinger}) on the spatial interval $I:=[x_{0}, x_{end}]$:
\begin{align} \label{schroedinger_IVP}
 \begin{cases}
		\eps^{2}\varphi^{\prime\prime}(x)+a(x)\varphi(x) = 0 \Comma \quad x\in I=[x_{0},x_{end}]\Comma\\
		\varphi(x_{0})=\varphi_{0} \Comma\\
        \varepsilon\varphi^{\prime}(x_{0})=\varphi_{1}\period
	\end{cases}
\end{align}
%
%with some $x_{0}\in\mathbb{R}$.
Before we explain the reformulation of this IVP, we introduce the following assumption.
\begin{Hypothesis}\label{hypothesis_A}
	Let $a\in C^{7}(I)$ be a fixed (real-valued) function, which satisfies $a(x)\geq a_{0} > 0$ for all $x\in I$. Further, let $0<\eps\leq \eps_{0}$ for some sufficiently small $\eps_{0}$.
\end{Hypothesis}
The basis of the transformation from \cite{Arnold2011WKBBasedSF} relies on the well-known WKB-ansatz (cf.\ \cite{Landau1985Q}) for the ODE (\ref{schroedinger}):
\begin{align} \label{wkb_ansatz}
	\varphi(x)\sim \exp\left(\frac{1}{\varepsilon}\sum_{p=0}^{\infty}\varepsilon^{p}\phi_{p}(x)\right)\Comma\quad \varepsilon\to 0\semicolon\quad \phi_{p}(x)\in\mathbb{C}\period
\end{align}
Here, the series in the exponential function is asymptotic w.r.t.\ $\varepsilon$ and typically divergent\footnote{See \cite{arnold2024optimally} for a class of coefficient functions $a$, which lead to convergent WKB series.}. To derive approximate solutions, it is therefore necessary to truncate the series after some finite order. By substituting the WKB-ansatz
into equation (\ref{schroedinger}), a comparison of $\varepsilon$-powers leads to the first three functions $\phi_{p}(x)$ as
\begin{align}
	\phi_{0}(x)&= \pm \ii \int_{x_{0}}^{x} \sqrt{a(y)} \, \mathrm{d}y \Comma \label{wkb_basis0}\\
	\phi_{1}(x)&= \ln(a(x)^{-\frac{1}{4}})-\ln(a(x_{0})^{-\frac{1}{4}}) \Comma \label{wkb_basis1}\\
	\phi_{2}(x)&= \mp \ii \int_{x_{0}}^{x} b(y) \, \mathrm{d}y \Comma \quad b(x):= -\frac{1}{2a(x)^{\frac{1}{4}}}\left(a(x)^{-\frac{1}{4}}\right)^{\prime\prime} \period \label{wkb_basis2}
\end{align}
Here, the two different signs $\pm$ and $\mp$ in (\ref{wkb_basis0}) and (\ref{wkb_basis2}) imply that there is always a pair of approximate solutions to the Schrödinger equation (\ref{schroedinger}). %Thus, the general approximate solution is then a linear combination of the two.
This corresponds to the fact that the second order ODE (\ref{schroedinger}) has two fundamental solutions. By truncating the series in (\ref{wkb_ansatz}) after $p = 2$, one obtains the second order (w.r.t.\ $\varepsilon$) asymptotic WKB approximation of $\varphi(x)$ as
\begin{align} \label{2ndorderwkb}
	\varphi_{2}^{WKB}(x) = \alpha\frac{\exp\left(\frac{\ii}{\varepsilon} \phi^{\varepsilon}(x)\right)}{a(x)^{\frac{1}{4}}} + \beta\frac{\exp\left(-\frac{\ii}{\varepsilon}\phi^{\varepsilon}(x)\right)}{a(x)^{\frac{1}{4}}} \Comma
\end{align}
with constants $\alpha,\beta\in\mathbb{C}$. The phase is
\begin{align} \label{phase}
	\phi^{\varepsilon}(x):= \int_{x_{0}}^{x}\left(\sqrt{a(y)}-\varepsilon^{2}b(y)\right) \, \mathrm{d}y \period
\end{align}
Note that Hypothesis~\ref{hypothesis_A} implies that $\phi^{\eps}$ is bounded on $I$. The WKB-based transformation now consists of two steps: First, using the notation
\begin{align} \label{Utrafo}
	U(x)=\begin{pmatrix}u_{1}(x)\\u_{2}(x)\end{pmatrix}:=\begin{pmatrix}a(x)^{\frac{1}{4}}\varphi(x)\\\frac{\varepsilon\left(a(x)^{\frac{1}{4}}\varphi(x)\right)^{\prime}}{\sqrt{a(x)}}\end{pmatrix} \Comma
\end{align}
the second order IVP (\ref{schroedinger_IVP}) can be reformulated as a system of first order differential equations:
\begin{align} \label{Usystem}
	\begin{cases}
		U^{\prime}(x) = \left[\frac{1}{\varepsilon}\mathbf{A}_{0}(x)+\varepsilon \mathbf{A}_{1}(x)\right]U(x) \Comma \quad x\in I=[x_{0},x_{end}]\Comma\\
		U(x_{0})=U_{0} \Comma
	\end{cases}
\end{align}
where the two matrices $\mathbf{A}_{0}$ and $\mathbf{A}_{1}$ are given by
\begin{align}
	\mathbf{A}_{0}(x):=\sqrt{a(x)}\begin{pmatrix}0 & 1 \\ -1 & 0\end{pmatrix}\semicolon\quad \mathbf{A}_{1}(x):=\begin{pmatrix}0 & 0 \\ 2b(x) & 0\end{pmatrix} \nonumber \period
\end{align}
Second, the first order system (\ref{Usystem}) is transformed by an additional change of variables
\begin{align} \label{Z-trafo}
	Z(x)=\begin{pmatrix}z_{1}(x) \\ z_{2}(x)\end{pmatrix}:=\exp\left(-\frac{\ii}{\varepsilon}\mathbf{\Phi}^{\varepsilon}(x)\right)\mathbf{P}U(x) \Comma
\end{align}
with the two matrices
\begin{align}
	\mathbf{P}:=\frac{1}{\sqrt{2}}\begin{pmatrix}\ii & 1 \\ 1 & \ii\end{pmatrix}\semicolon \quad \mathbf{\Phi}^{\varepsilon}(x):=\begin{pmatrix}\phi^{\varepsilon}(x) & 0 \\ 0 & -\phi^{\varepsilon}(x) \end{pmatrix} \nonumber \period
\end{align}
Here, the function $\phi^{\varepsilon}$ is the phase of the second order WKB approximation (\ref{2ndorderwkb}) as defined in (\ref{phase}).
%leads to
%\begin{align}
%\begin{cases}
%Y^{\prime}(x) = \frac{\ii}{\varepsilon}D^{\varepsilon}(x)Y(x)+\varepsilon N(x)Y(x) \Comma \quad x\in\\
%Y(x_{0})=Y_{I}
%\end{cases}
%\end{align}
%with
%\begin{align}
%D^{\varepsilon}(x)=\begin{pmatrix}\sqrt{a(x)}-\varepsilon^{2}\beta(x) & 0 \\ 0 & -\sqrt{a(x)}+\varepsilon^{2}\beta(x)\end{pmatrix}\semicolon\quad N(x):=\begin{pmatrix}0 & \beta(x) \\ \beta(x) & 0\end{pmatrix}\period
%\end{align}
%A final transformation for the elimination of leading oscillations by defining the diagonal matrix
%\begin{align}
%	\Phi^{\varepsilon}(x):=\begin{pmatrix}\phi^{\varepsilon}(x) & 0 \\ 0 & -\phi^{\varepsilon}(x) \end{pmatrix}
%\end{align}
%and a change of unknown
%\begin{align}
%	Z(x)=\begin{pmatrix}z_{1}(x) \\ z_{2}(x)\end{pmatrix}:=\exp\left(-\frac{\ii}{\varepsilon}\Phi^{\varepsilon}(x)\right)Y(x) \comma
%\end{align}
The resulting system for $Z$ then reads
\begin{align} \label{Zsystem}
	\begin{cases}
		Z^{\prime}(x) = \varepsilon \mathbf{N}^{\varepsilon}(x)Z(x) \Comma \quad x\in I=[x_{0},x_{end}]\Comma\\
		Z(x_{0})=Z_{0}=\mathbf{P}U_{0} \Comma
	\end{cases}
\end{align}
where $\mathbf{N}^{\varepsilon}(x)$ is a (Hermitian) matrix with only off-diagonal non-zero entries:
\begin{align}\label{N_off_diagonal_entries}
	N_{1,2}^{\varepsilon}(x)=b(x)\operatorname{e}^{-\frac{2\ii}{\varepsilon}\phi^{\varepsilon}(x)}\Comma\quad N_{2,1}^{\varepsilon}(x)=b(x)\operatorname{e}^{\frac{2\ii}{\varepsilon}\phi^{\varepsilon}(x)} \period
\end{align}
As stated in \cite{Arnold2011WKBBasedSF}, the IVP (\ref{Zsystem}) admits a unique solution with the following estimates:
\begin{Proposition}{\cite[Theorem 3.1]{Arnold2011WKBBasedSF}}\label{proposition_Z_solution}
	Let Hypothesis \ref{hypothesis_A} be satisfied. Then the problem (\ref{Zsystem}) has a unique solution $Z\in C^{6}(I)$ with the explicit form
	\begin{align}\label{picard_limit}
		Z(x)=Z_{0}+\sum_{p=1}^{\infty}\eps^{p}\mathbf{M}_{p}^{\eps}(x;x_{0})Z_{0}\Comma
	\end{align}
	where the matrices $\mathbf{M}_{p}^{\eps}$, $p\geq 1$ are given by
	\begin{align}
		\mathbf{M}_{p}^{\eps}(\eta;\xi)&=\int_{\xi}^{\eta}\int_{\xi}^{y_{1}}\cdots\int_{\xi}^{y_{p-1}}\mathbf{N}^{\eps}(y_{1})\cdots\mathbf{N}^{\eps}(y_{p})\,\mathrm{d}y_{p}\cdots\mathrm{d}y_{1}\Comma\nonumber\\
		\mathbf{M}_{p}^{\eps}(\eta;\xi)&=\int_{\xi}^{\eta}\mathbf{N}^{\eps}(y)\mathbf{M}_{p-1}^{\eps}(y;\xi)\,\mathrm{d}y\Comma\quad \mathbf{M}_{0}^{\eps}=\mathbf{I}\period\label{M_p_def}
	\end{align}
	Here, $\mathbf{I}$ denotes the $2\times 2$ identity matrix. Moreover we have the estimates
	\begin{align}
		\lVert Z-Z_{0}\rVert_{L^{\infty}(I)} \leq C\eps^{2}\Comma\quad \lVert Z^{\prime}\rVert_{L^{\infty}(I)}\leq C\eps\Comma \quad \lVert Z^{\prime\prime}\rVert_{L^{\infty}(I)} \leq C \Comma
	\end{align}
	with a constant $C>0$ independent of $\eps$.
\end{Proposition}
\anton{Due to the oscillatory entries in its system matrix $\mathbf{N}^{\varepsilon}(x)$, see \eqref{N_off_diagonal_entries}, the solution $Z$ of IVP (\ref{Zsystem}) is still oscillatory, and it typically has twice the frequency compared to the wave function $\varphi$, see Figure 2.1 in \cite{Arnold2011WKBBasedSF}. But according to Proposition~\ref{proposition_Z_solution}, the solution $Z$ of IVP (\ref{Zsystem}) oscillates around the initial value $Z_{0}$ with a decreasing amplitude of at most $\mathcal{O}(\varepsilon^{2})$, and this makes the $Z$-oscillations decreasingly relevant as $\varepsilon\to0$. }
In \cite{Arnold2011WKBBasedSF}, this fact was the motivation \anton{and essential justification} for the construction of a uniformly correct scheme in $\eps$, and it shall also be the motivation for the scheme we develop in Section~\ref{chap:construction_scheme}.
The idea is that, instead of solving IVP (\ref{schroedinger_IVP}), we aim to solve the transformed problem (\ref{Zsystem}). Indeed, since the transformation (\ref{Z-trafo}) eliminated the dominant oscillations, the IVP (\ref{Zsystem}) can be solved numerically on a coarse grid $\{x_{n}\}$. Then, the originally desired solution $\varphi$ can be obtained by the inverse transformation
\begin{align}\label{U_inverse}
	U(x)=\mathbf{P}^{-1}\exp\left(\frac{\ii}{\varepsilon}\mathbf{\Phi}^{\varepsilon}(x)\right)Z(x)
\end{align}
and $\varphi(x)=a(x)^{-\frac{1}{4}}u_{1}(x)$.

%%%%%%%%%%%%%%%%%%%%%%%%%%%%%%%%%%%%%%%%%%%%%%%%%%%%%%%%

\section{Construction of the numerical method}\label{chap:construction_scheme}
The aim of this section is to construct a third order (w.r.t.\ the step size) one-step scheme for solving the Schrödinger equation-IVP (\ref{schroedinger_IVP}). To this end, we build upon the WKB-based transformation from Section~\ref{chap:WKB-trafo}, in the same way as it was done in \cite{Arnold2011WKBBasedSF}. That is, instead of solving (\ref{schroedinger_IVP}) directly, we will solve the transformed problem (\ref{Zsystem}). %That said, the resulting method will have both a \textit{WKB-order} (w.r.t.\ $\varepsilon$; referring to the used cut-off in the asymptotic expansion (\ref{wkb_ansatz})) and a \textit{numerical order} (w.r.t.\ the step size $h$; referring to the convergence order).

We now consider a discretization $\left\{x_{0},x_{1},...,x_{N}\right\}$ of the interval $I=[x_{0}, x_{end}]$ and set $h:=\max_{1\leq n\leq N}\vert x_{n}-x_{n-1}\vert$ as the maximum step size. Further, throughout this whole section we will use the abbreviations $\xi:=x_{n}$ and $\eta:=x_{n+1}$.

The development of the one-step method relies on deriving specific quadrature rules for the matrix-valued integrals (\ref{M_p_def}), with error estimates depending on the two small parameters $\varepsilon$ and $h$. This leads us to introduce the following notation, which is analogous to the usual `big-O' notation for a single parameter:
\begin{Definition}
    Let $V$ be a vector space (over $\mathbb{C}$) with norm $\lVert\cdot\rVert$ and let $0<\varepsilon_{0},h_{0}<1$. Consider two functions $f:(0,\varepsilon_{0})\times(0,h_{0})\to V$ and $g:(0,\varepsilon_{0})\times(0,h_{0})\to \mathbb{R}$. We say that $f(\varepsilon,h)=\mathcal{O}_{\varepsilon,h}(g(\varepsilon,h))$, if and only if there exists a constant $C>0$, such that $\lVert f(\varepsilon,h)\rVert\leq C g(\varepsilon,h)$, for any $(\varepsilon,h)\in(0,\varepsilon_{0})\times(0,h_{0})$.
\end{Definition}
Now, in order to derive a numerical scheme of order $P$ (w.r.t.\ the step size $h$) for solving the IVP (\ref{Zsystem}), we shall start with the $P$-th order Picard approximation
\begin{align} \label{P_picard}
	Z(\eta)\approx Z(\xi)+\sum_{p=1}^{P}\eps^{p}\mathbf{M}_{p}^{\eps}(\eta;\xi)Z(\xi)
\end{align}
of (\ref{picard_limit}), which corresponds to making a single step within the grid, namely, from $\xi$ to $\eta$. According to \cite[Eq.\ (2.17)]{Arnold2011WKBBasedSF}, the remainder of this truncated series is of order $\mathcal{O}_{\varepsilon,h}(\eps^{P+1}h^{P}\min(\eps,h))$, i.e., we have
\begin{align}\label{picard_remainder}
	\left\lVert \sum_{p=P+1}^{\infty}\eps^{p}\mathbf{M}_{p}^{\eps}(\eta;\xi)\right\rVert_{\infty}\leq C\eps^{P+1}h^{P}\min(\eps,h)\Comma
\end{align}
where the constant $C\geq 0$ is independent of $\eps$ and $h$. Here, $\lVert\cdot\rVert_{\infty}$ denotes the $\infty$-matrix norm in $\mathbb{C}^{2\times 2}$. Since the entries (\ref{N_off_diagonal_entries}) of the matrix $\mathbf{N}^{\varepsilon}(y)$ are highly oscillatory, the matrices $\mathbf{M}_{p}^{\eps}(\eta;\xi)$ in (\ref{P_picard}) and (\ref{picard3}) involve (iterated) oscillatory integrals and can therefore not be computed exactly in general. To design suitable quadrature rules for the oscillatory integrals $\mathbf{M}_{p}^{\eps}(\eta;\xi)$ we shall \anton{follow the strategy from \cite{Arnold2011WKBBasedSF}} and use techniques \anton{motivated by} the \textit{asymptotic method} from \cite{Iserles2006HighlyOscill}, which mainly relies on integration by parts. \anton{But in contrast to \cite{Iserles2006HighlyOscill} we have to cope here with two small parameters, and they must appear in the desired error estimates.} If we want to keep the error order (\ref{picard_remainder}) also when approximating the oscillatory integrals, the desired error order (w.r.t.\ $\eps$ and $h$) for the approximation of the integral $\mathbf{M}_{p}^{\eps}$ is $\mathcal{O}_{\varepsilon,h}(\eps^{P+1-p}h^{P}\min(\eps,h))$.

To derive a third order (w.r.t.\ $h$) scheme for (\ref{Zsystem}), we only take into account the first three terms of the sum in (\ref{P_picard}). Indeed, using estimate (\ref{picard_remainder}) we then have
\begin{align}\label{picard3}
	Z(\eta)&=\left[\mathbf{I}+\varepsilon \mathbf{M}_{1}^{\varepsilon}(\eta;\xi)+\varepsilon^{2} \mathbf{M}_{2}^{\varepsilon}(\eta;\xi)+\varepsilon^{3} \mathbf{M}_{3}^{\varepsilon}(\eta;\xi)\right]Z(\xi)\nonumber\\
    &\quad+ \mathcal{O}_{\varepsilon,h}(\varepsilon^{4}h^{3}\min(\varepsilon,h))\period
\end{align}
The goal of the following three subsections is to construct approximations for the matrices $\mathbf{M}_{1}^{\varepsilon}$, $\mathbf{M}_{2}^{\varepsilon}$, and $\mathbf{M}_{3}^{\varepsilon}$, respectively. Unfortunately, achieving the above mentioned desired approximation error order for each matrix is not always possible, as some terms involved in the computations constrain the maximum achievable order w.r.t.\ $\eps$ (this was also the case in the construction of the second order scheme in \cite{Arnold2011WKBBasedSF}). However, we will be able to construct approximations for $\mathbf{M}_{p}^{\eps}$, $p=1,2,3$, which lead to a third order scheme with a local discretization error of the order $\mathcal{O}_{\eps,h}(\eps^{3}h^{4}\max(\eps,h))$.
\subsection{Approximation of the matrix \texorpdfstring{$\mathbf{M}_{1}^{\varepsilon}$}{M1}}\label{M1}
We start with the approximation of the integral
\begin{align}
	\mathbf{M}_{1}^{\varepsilon}(\eta;\xi)&=\int_{\xi}^{\eta}\mathbf{N}^{\varepsilon}(y)\,\mathrm{d}y\period\nonumber
\end{align}
As we already mentioned above, the desired error order for this approximation is $\mathcal{O}_{\varepsilon,h}(\eps^{3}h^{3}\min(\eps,h))$.

For the phase $\phi^{\eps}$ let us suppress the superscript $\eps$ and just write $\phi$ from now on. Further, we denote with
\begin{align}
    s_{n}:=\phi(\eta)-\phi(\xi)=\mathcal{O}(h)\Comma\quad h\to 0\Comma\nonumber
\end{align}
the phase increments and introduce the abbreviations
\begin{align}\label{b-compuations}
	b_{0}(x):=\frac{b(x)}{2\phi^{\prime}(x)}\Comma\quad b_{p+1}(x):=\frac{b_{p}^{\prime}(x)}{2\phi^{\prime}(x)}\Comma \quad p\in\mathbb{N}_{0}\period
\end{align}
Note that Hypothesis \ref{hypothesis_A} implies that for $0<\eps\leq \eps_{0}$, where $\eps_{0}$ is chosen sufficiently small, it holds $\phi'(x)\geq C_{\phi}>0$ for some constant $C_{\phi}$ and for all $x\in I$, see (\ref{phase}). Hence both the function $b$ and the functions $b_{p}$ are uniformly bounded (w.r.t.\ $\eps$) on the interval $I$.
Moreover, we define the auxiliary function 
\begin{align}
	h_{p}(x):=\ee^{\ii x}-\sum_{k=0}^{p-1}\frac{(\ii x)^{k}}{k!}\Comma \quad p\in\mathbb{N}_{0}\nonumber
\end{align}
(for $p=0$ the sum drops). In the following we will frequently use that
\begin{align}
    h_{p-1}(f(x))=-\ii (f'(x))^{-1}\frac{\dd}{\dd x}h_{p}(f(x)))\Comma\quad p\in\mathbb{N}\Comma\nonumber
\end{align}
for any differentiable function $f$ on $I$ and any $x\in I$ with $f'(x)\neq0$, and
\begin{align}\label{h_p_order}
		h_{p}(x)=\mathcal{O}(\min(x^{p},x^{p-1}))\Comma \quad x\geq 0\Comma \quad p\in\mathbb{N}\period
\end{align}

The following lemma provides an approximation to $\mathbf{M}_{1}^{\varepsilon}$, with an error order $\mathcal{O}_{\varepsilon,h}(\eps^{P}h^{\widetilde{P}}\min(\eps,h))$, for any $P\in\mathbb{N}_{0}$, $\widetilde{P}\in\mathbb{N}$.
\begin{Lemma}\label{lemma_m1}
Let Hypothesis \ref{hypothesis_A} be satisfied. For any $P\in\mathbb{N}_{0}$ and $\widetilde{P}\in\mathbb{N}$ define
	\begin{align}\label{Q_1}
		Q_{1}^{P,\widetilde{P}}(\eta,\xi)&:=-\sum_{p=1}^{P}(\ii\eps)^{p}\left(b_{p-1}(\eta)\ee^{\frac{2\ii}{\eps}\phi(\eta)}-b_{p-1}(\xi)\ee^{\frac{2\ii}{\eps}\phi(\xi)}\right)\nonumber\\
		&\quad-\ee^{\frac{2\ii}{\eps}\phi(\xi)}\sum_{p=1}^{\widetilde{P}}(\ii\eps)^{p+P}b_{p+P-1}(\eta)h_{p}\left(\frac{2}{\eps}s_{n}\right)
	\end{align}
(for $P=0$ the first sum drops) and
	\begin{align}\label{Q1}
		\mathbf{Q}_{1}^{P,\widetilde{P}}(\eta,\xi):=\begin{pmatrix}
			0 & \overline{Q_{1}^{P,\widetilde{P}}(\eta,\xi)} \\
			Q_{1}^{P,\widetilde{P}}(\eta,\xi) & 0 
		\end{pmatrix}\period
	\end{align}
	Then there exists $C\geq 0$ independent of $\varepsilon\in(0,\eps_{0}]$, $h$, and $n$, such that
	\begin{align}\label{mixed expansion}
		\lVert \mathbf{M}_{1}^{\eps}(\eta,\xi)-\mathbf{Q}_{1}^{P,\widetilde{P}}(\eta,\xi)\rVert_{\infty} \leq C\eps^{P}h^{\widetilde{P}}\min(\eps,h)\period
	\end{align}
%	\begin{align}\label{mixed expansion}
%		m_{1}^{\eps}(\eta,\xi)&=-\sum_{p=1}^{P}(\ii\eps)^{p}\left(b_{p-1}(\eta)\ee^{\frac{2\ii}{\eps}\phi(\eta)}-b_{p-1}(\xi)\ee^{\frac{2\ii}{\eps}\phi(\xi)}\right)\nonumber\\
%		&\quad-\ee^{\frac{2\ii}{\eps}\phi(\xi)}\sum_{p=1}^{P}(\ii\eps)^{p+P}b_{p+P-1}(\eta)h_{p}\left(\frac{2}{\eps}(\phi(\eta)-\phi(\xi))\right)+\mathcal{O}(\eps^{P}h^{P}\min(\eps,h))\Comma
%	\end{align}
%	%
%	with $h_{p}$ defined as
%	\begin{align}
%		h_{p}(\eta):=\ee^{\ii\eta}-\sum_{k=0}^{p-1}\frac{(\ii\eta)^{k}}{k!}
%	\end{align}
	%where the first sum corresponds to $P$ st.a.m.\ steps and the second sum corresponds to $P$ sh.a.m.\ steps after that.
\end{Lemma}
\begin{proof}
	Recall from (\ref{N_off_diagonal_entries}) that $\mathbf{N}^{\varepsilon}(y)$ is a Hermitian matrix with only off-diagonal non-zero entries, namely  $N_{2,1}^{\varepsilon}(y)=b(y)\operatorname{e}^{\frac{2\ii}{\varepsilon}\phi(y)}=\overline{N_{1,2}^{\varepsilon}(y)}$. Hence, it is sufficient to prove that $\lVert m_{1}^{\eps}(\eta,\xi)-Q_{1}^{P,\widetilde{P}}(\eta,\xi)\rVert_{\infty} \leq C\eps^{P}h^{\widetilde{P}}\min(\eps,h)$, where $m_{1}^{\varepsilon}(\eta,\xi):=\left(\mathbf{M}_{1}^{\varepsilon}(\eta;\xi)\right)_{2,1}$. To this end, we start by expanding the integral $m_{1}^{\varepsilon}(\eta,\xi)$ by making $P$ steps of the so-called \textit{asymptotic method (AM)} for oscillatory integrals (cf.\ \cite{Iserles2006HighlyOscill}), which relies on repeated integration by parts:
	\begin{align}\label{stam}
		m_{1}^{\varepsilon}(\eta,\xi)&=\int_{\xi}^{\eta}b(y)\ee^{\frac{2\ii}{\eps}\phi(y)}\,\mathrm{d}y\nonumber\\
		&=-(\ii\eps)\int_{\xi}^{\eta}b_{0}(y)\frac{\dd}{\dd y}\left[\ee^{\frac{2\ii}{\eps}\phi(y)}\right]\,\dd y\nonumber\\
		&=-(\ii\eps)\left[b_{0}(y)\ee^{\frac{2\ii}{\eps}\phi(y)}\right]_{\xi}^{\eta}+(\ii\eps)\int_{\xi}^{\eta}b_{0}^{\prime}(y)\ee^{\frac{2\ii}{\eps}\phi(y)}\,\dd y\nonumber\\
		%&=-(\ii\eps)\left[b_{0}(y)\ee^{\frac{2\ii}{\eps}\phi(y)}\right]_{\xi}^{\eta}-(\ii\eps)^{2}\left[b_{1}(y)\ee^{\frac{2\ii}{\eps}\phi(y)}\right]_{\xi}^{\eta}+(\ii\eps)^{2}\int_{\xi}^{\eta}b_{1}^{\prime}(y)\ee^{\frac{2\ii}{\eps}\phi(y)}\,\dd y\nonumber\\
		&=-\sum_{p=1}^{P}(\ii\eps)^{p}\left[b_{p-1}(y)\ee^{\frac{2\ii}{\eps}\phi(y)}\right]^{\eta}_{\xi}+T_{P}^{\eps}(\eta,\xi)\period
	\end{align}
	Here, $T_{P}^{\eps}(\eta,\xi)$ is the remaining integral after the last integration by parts. It is evident from (\ref{stam}) that each step of integration by parts increases the $\eps$-order of the remaining integral by one, since the functions $b_{p}$ are uniformly bounded w.r.t.\ $\eps$. Next we approximate $T_{P}^{\eps}(\eta,\xi)$ by making $\widetilde{P}$ steps of the so-called \textit{shifted asymptotic method (SAM)} (see \cite{Arnold2011WKBBasedSF}), which also relies on repeated integration by parts. However, in contrast to the AM, the idea of the SAM is to shift the oscillatory factor such that an additional zero within the integration interval is created (here we create it at $y=\xi$). This increases the $h$-order of the remainder by one in each step:
	\begin{align}\label{sham}
		T_{P}^{\eps}(\eta,\xi)&=(\ii\eps)^{P}\int_{\xi}^{\eta}b_{P-1}^{\prime}(y)\ee^{\frac{2\ii}{\eps}\phi(y)}\,\dd y\nonumber\\
		&=-(\ii\eps)^{P+1}\ee^{\frac{2\ii}{\eps}\phi(\xi)}\int_{\xi}^{\eta}b_{P}(y)\frac{\dd}{\dd y}\left[h_{1}\left(\frac{2}{\eps}(\phi(y)-\phi(\xi))\right)\right]\,\dd y\nonumber\\
		&=-(\ii\eps)^{P+1}\ee^{\frac{2\ii}{\eps}\phi(\xi)}\Bigg\{b_{P}(\eta)h_{1}\left(\frac{2}{\eps}s_{n}\right)\nonumber\\
        &\quad\quad\quad\quad\quad\quad\quad\quad\quad-\int_{\xi}^{\eta}b_{P}^{\prime}(y)h_{1}\left(\frac{2}{\eps}(\phi(y)-\phi(\xi))\right)\,\dd y\Bigg\}\nonumber\\
		&=-(\ii\eps)^{P+1}\ee^{\frac{2\ii}{\eps}\phi(\xi)}\Bigg\{b_{P}(\eta)h_{1}\left(\frac{2}{\eps}s_{n}\right)+(\ii\eps)b_{P+1}(\eta)h_{2}\left(\frac{2}{\eps}s_{n}\right)\nonumber\\
		&\quad\quad\quad\quad\quad\quad\quad\quad\quad-(\ii\eps)\int_{\xi}^{\eta}b_{P+1}^{\prime}(y)h_{2}\left(\frac{2}{\eps}(\phi(y)-\phi(\xi))\right)\,\dd y\Bigg\}\nonumber\\
		&=-(\ii\eps)^{P+1}\ee^{\frac{2\ii}{\eps}\phi(\xi)}\sum_{p=1}^{\widetilde{P}}(\ii\eps)^{p-1}b_{p+P-1}(\eta)h_{p}\left(\frac{2}{\eps}s_{n}\right)\nonumber\\
        &\quad+\mathcal{O}_{\varepsilon,h}(\eps^{P}h^{\widetilde{P}}\min(\eps,h))
	\end{align}
	(for $P=0$ we set $b_{-1}':=b$). Here, we used for the remainder integral in the last equation $h_{\widetilde{P}}\left(\frac{2}{\eps}(\phi(y)-\phi(\xi))\right)=\mathcal{O}_{\varepsilon,h}(\min(h^{\widetilde{P}}\eps^{-\widetilde{P}},h^{\widetilde{P}-1}\eps^{-(\widetilde{P}-1)}))$, which is a consequence of (\ref{h_p_order}). This concludes the proof.
\end{proof}
\begin{Remark}
    A drawback of using quadrature (\ref{Q_1}) is that one has to provide explicit formulas for the functions $b_{p}$. Indeed, since these functions are defined in a recursive manner by differentiating a quotient of two functions, the number of terms involved in the formulas grows fast with $p$. For an efficient implementation we recommend to express the functions $b_{p}$ only through the functions $b^{(k)}$, $k=0,\dots,p$, and $\phi^{(k)}$, $k=1,\dots, p+1$, as this keeps the formulas much shorter compared to when expressing $b_{p}$ through $a$ and its derivatives $a^{(k)}$, $k=1,\dots, p+2$.
\end{Remark}
In the following, almost all approximations of occurring integrals will be derived by using the AM or the SAM. Recall that each step of the AM in (\ref{stam}) increased the $\eps$-order of the remainder by one, whereas each step of the SAM in (\ref{sham}) increased the $h$-order by one. Thus, by combining both methods appropriately, one might hope to be able to achieve the desired mixed (w.r.t.\ $\eps$ and $h$) error order also for the approximations of the matrices $\mathbf{M}_{2}^{\eps}$ and $\mathbf{M}_{3}^{\eps}$. Unfortunately, this will not be the case as we will see in the following two subsections. Finally, we note that this strategy of combining the AM with the SAM was already used in \cite{Arnold2011WKBBasedSF} for the construction of their second order (w.r.t.\ $h$) scheme.
\subsection{Approximation of the matrix \texorpdfstring{$\mathbf{M}_{2}^{\eps}$}{M2}}\label{M2}
Next, according to (\ref{picard3}), we would like to approximate the integral
\begin{align}\label{M2_matrix}
    \mathbf{M}_{2}^{\eps}(\eta;\xi)=\int_{\xi}^{\eta}\mathbf{N}^{\eps}(y)\mathbf{M}_{1}^{\eps}(y;\xi)\,\dd y
\end{align}
with the error order $\mathcal{O}_{\varepsilon,h}(\eps^{2}h^{3}\min(\eps,h))$. But this will not be possible since the matrix product $\mathbf{N}^{\eps}(y)\mathbf{M}_{1}^{\eps}(y;\xi)$ contains terms which involve non-oscillatory integrals. Indeed, the basis for the AM to increase the $\varepsilon$-order in each step of (\ref{stam}) was the oscillatory factor $\ee^{\frac{2\ii}{\eps}\phi(y)}$ of the integrand. However, for non-oscillatory integrals, it is only feasible to derive approximations with arbitrary $h$-order, but not with arbitrary $\eps$-order. Nonetheless, in this subsection, we will derive an approximation for $\mathbf{M}_{2}^{\eps}$, which ultimately leads to a third order scheme (w.r.t.\ $h$) that is uniformly correct w.r.t.\ $\eps$ -- in fact with a local discretization error of the order $\mathcal{O}_{\eps,h}(\eps^{3}h^{4}\max(\eps,h))$. This will be achieved by employing a sufficiently accurate quadrature formula for the non-oscillatory integrals. To this end, let us introduce the notation
\begin{align}
	Q_{S}[f](\eta,\xi):=\frac{\eta-\xi}{6}\left(f(\xi)+4f\left(\frac{\xi+\eta}{2}\right)+f(\eta)\right)\Comma\nonumber
\end{align}
which denotes Simpson's rule applied to the integral $\int_{\xi}^{\eta}f(y)\,\mathrm{d}y$, admitting an error of the order $\mathcal{O}(h^{5})$.

Since $\mathbf{N}^{\eps}$ and $\mathbf{M}_{1}^{\eps}$ are off-diagonal Hermitian matrices, the matrix $\mathbf{M}_{2}^{\eps}$ is diagonal and the entries are conjugate of one another. Hence, we shall in the following only study one entry, namely,
\begin{align}\label{m2_definition}
    m_{2}^{\eps}(\eta,\xi):=\left(\mathbf{M}_{2}^{\eps}(\eta;\xi)\right)_{1,1}=\int_{\xi}^{\eta}\left(\mathbf{N}^{\eps}(y)\right)_{1,2}\left(\mathbf{M}_{1}^{\eps}(y;\xi)\right)_{2,1}\,\dd y\period    
\end{align}
Here, it seems reasonable to insert for $\left(\mathbf{M}_{1}^{\eps}(y;\xi)\right)_{2,1}=m_{1}^{\varepsilon}(y,\xi)$ the approximation from Lemma \ref{lemma_m1} with $P=\widetilde{P}=2$, namely $Q_{1}^{2,2}(y,\xi)$, as the order of its approximation error increases exactly to $\mathcal{O}_{\varepsilon,h}(\eps^{2}h^{3}\min(\eps,h))$, when integrating over the subinterval $[\xi,\eta]$. From (\ref{m2_definition}), (\ref{N_off_diagonal_entries}), and (\ref{mixed expansion}) we have:
\begin{align}\label{m_2_approximation}
	m_{2}^{\eps}(\eta,\xi)&=\int_{\xi}^{\eta}b(y)\ee^{-\frac{2\ii}{\eps}\phi(y)}\left\{Q_{1}^{2,2}(y,\xi)+\mathcal{O}_{\varepsilon,h}(\eps^{2}h^{2}\min(\eps,h))\right\}\,\dd y\nonumber\\
	&=\int_{\xi}^{\eta}b(y)\ee^{-\frac{2\ii}{\eps}\phi(y)}\Bigg\{-(\ii\eps)\left(b_{0}(y)\ee^{\frac{2\ii}{\eps}\phi(y)}-b_{0}(\xi)\ee^{\frac{2\ii}{\eps}\phi(\xi)}\right)\nonumber\\
    &\quad\quad\quad\quad\quad\quad\quad\quad\quad-(\ii\eps)^{2}\left(b_{1}(y)\ee^{\frac{2\ii}{\eps}\phi(y)}-b_{1}(\xi)\ee^{\frac{2\ii}{\eps}\phi(\xi)}\right)\nonumber\\
    &\quad\quad\quad\quad\quad\quad\quad\quad\quad-(\ii\eps)^{3}\ee^{\frac{2\ii}{\eps}\phi(\xi)}b_{2}(y)h_{1}\left(\frac{2}{\eps}(\phi(y)-\phi(\xi))\right)\nonumber\\
    &\quad\quad\quad\quad\quad\quad\quad\quad\quad-(\ii\eps)^{4}\ee^{\frac{2\ii}{\eps}\phi(\xi)}b_{3}(y)h_{2}\left(\frac{2}{\eps}(\phi(y)-\phi(\xi))\right)\nonumber\\
    &\quad\quad\quad\quad\quad\quad\quad\quad\quad+\mathcal{O}_{\varepsilon,h}(\eps^{2}h^{2}\min(\eps,h))\Bigg\}\,\dd y\nonumber\\
	&=J_{1}+J_{2}+J_{3}+J_{4}+J_{5}+J_{6}+\mathcal{O}_{\varepsilon,h}(\eps^{2}h^{3}\min(\eps,h))\Comma
\end{align}
where
\begin{align}
	J_{1}&:=-(\ii\eps)\int_{\xi}^{\eta}b(y)b_{0}(y)\,\dd y = -(\ii\eps)Q_{S}[bb_{0}](\eta,\xi)+\mathcal{O}_{\varepsilon,h}(\varepsilon h^5) \Comma\label{J1}\\
	J_{2}&:=(\ii\eps)b_{0}(\xi)\ee^{\frac{2\ii}{\eps}\phi(\xi)}\int_{\xi}^{\eta}b(y)\ee^{-\frac{2\ii}{\eps}\phi(y)}\,\dd y =(\ii\eps)b_{0}(\xi)\ee^{\frac{2\ii}{\eps}\phi(\xi)}\overline{m_{1}^{\eps}(\eta,\xi)}\Comma\nonumber\\
	J_{3}&:=-(\ii\eps)^{2}\int_{\xi}^{\eta}b(y)b_{1}(y)\,\dd y = -(\ii\eps)^{2}Q_{S}[bb_{1}](\eta,\xi)+\mathcal{O}_{\varepsilon,h}(\varepsilon^{2} h^{5}) \Comma\label{J3}\\
	J_{4}&:=(\ii\eps)^{2}b_{1}(\xi)\ee^{\frac{2\ii}{\eps}\phi(\xi)}\int_{\xi}^{\eta}b(y)\ee^{-\frac{2\ii}{\eps}\phi(y)}\,\dd y = (\ii\eps)^{2}b_{1}(\xi)\ee^{\frac{2\ii}{\eps}\phi(\xi)}\overline{m_{1}^{\eps}(\eta,\xi)}\Comma\nonumber\\
	J_{5}&:=-(\ii\eps)^{3}\ee^{\frac{2\ii}{\eps}\phi(\xi)}\int_{\xi}^{\eta}b(y)b_{2}(y)\ee^{-\frac{2\ii}{\eps}\phi(y)}h_{1}\left(\frac{2}{\eps}(\phi(y)-\phi(\xi))\right)\,\dd y \Comma\nonumber\\
	J_{6}&:=-(\ii\eps)^{4}\ee^{\frac{2\ii}{\eps}\phi(\xi)}\int_{\xi}^{\eta}b(y)b_{3}(y)\ee^{-\frac{2\ii}{\eps}\phi(y)}h_{2}\left(\frac{2}{\eps}(\phi(y)-\phi(\xi))\right)\,\dd y\period\nonumber
\end{align}
Our next goal is to find suitable approximations for the integrals $J_{2}$, $J_{4}$, $J_{5}$, and $J_{6}$. Using Lemma~\ref{lemma_m1} with $P=1$ and $\widetilde{P}=3$ we obtain by using $\overline{h_{p}(x)}=h_{p}(-x)$:
\begin{align}\label{J2}
J_{2}
%&=(\ii\eps)b_{0}(\xi)\ee^{\frac{2\ii}{\eps}\phi(\xi)}\int_{\xi}^{\eta}b(y)\ee^{-\frac{2\ii}{\eps}\phi(y)}\,\dd y\nonumber\\
%	&=(\ii\eps)^{2}b_{0}(\xi)\ee^{\frac{2\ii}{\eps}\phi(\xi)}\int_{\xi}^{\eta}b_{0}(y)\frac{\dd}{\dd y}\left[\ee^{-\frac{2\ii}{\eps}\phi(y)}\right]\,\dd y\nonumber\\
%	&=(\ii\eps)^{2}b_{0}(\xi)\ee^{\frac{2\ii}{\eps}\phi(\xi)}\left[b_{0}(y)\ee^{-\frac{2\ii}{\eps}\phi(y)}\right]_{\xi}^{\eta}-(\ii\eps)^{2}b_{0}(\xi)\ee^{\frac{2\ii}{\eps}\phi(\xi)}\int_{\xi}^{\eta}b_{0}^{\prime}(y)\ee^{-\frac{2\ii}{\eps}\phi(y)}\,\dd y\nonumber\\
	&=(\ii\eps)^{2}b_{0}(\xi)\ee^{\frac{2\ii}{\eps}\phi(\xi)}\left[b_{0}(y)\ee^{-\frac{2\ii}{\eps}\phi(y)}\right]_{\xi}^{\eta}-(\ii\eps)^{3}b_{0}(\xi)\Bigg\{b_{1}(\eta)h_{1}\left(-\frac{2}{\eps}s_{n}\right)\nonumber\\
	&\quad-(\ii\eps)b_{2}(\eta)h_{2}\left(-\frac{2}{\eps}s_{n}\right)+(\ii\eps)^{2}b_{3}(\eta)h_{3}\left(-\frac{2}{\eps}s_{n}\right)\Bigg\}\nonumber\\
    &\quad+\mathcal{O}_{\varepsilon,h}(\eps^{2}h^{3}\min(\eps,h)).
\end{align}
Similarly, applying Lemma~\ref{lemma_m1} with $P=0$ and $\widetilde{P}=3$ we find that
\begin{align}\label{J4}
	J_{4}
	%&=(\ii\eps)^{2}b_{1}(\xi)\ee^{\frac{2\ii}{\eps}\phi(\xi)}\int_{\xi}^{\eta}b(y)\ee^{-\frac{2\ii}{\eps}\phi(y)}\,\dd y\nonumber\\
	&=(\ii\eps)^{3}b_{1}(\xi)\Bigg\{b_{0}(\eta)h_{1}\left(-\frac{2}{\eps}s_{n}\right)-(\ii\eps)b_{1}(\eta)h_{2}\left(-\frac{2}{\eps}s_{n}\right)\nonumber\\
	&\quad+(\ii\eps)^{2}b_{2}(\eta)h_{3}\left(-\frac{2}{\eps}s_{n}\right)\Bigg\}+\mathcal{O}_{\varepsilon,h}(\eps^{2}h^{3}\min(\eps,h))\period
\end{align}
For the approximation of the integral $J_{5}$, we begin by using the simple identity
\begin{align}\label{help1}
	\ee^{-\frac{2\ii}{\eps}\phi(y)}h_{1}\left(\frac{2}{\eps}(\phi(y)-\phi(\xi))\right)=-\ee^{-\frac{2\ii}{\eps}\phi(\xi)}h_{1}\left(\frac{2}{\eps}(\phi(\xi)-\phi(y))\right)\period
\end{align}
Then, by making two SAM-steps, we derive
\begin{align}\label{J5}
	J_{5}
	%&=-(\ii\eps)^{3}\ee^{\frac{2\ii}{\eps}\phi(\xi)}\int_{\xi}^{\eta}b(y)b_{2}(y)\ee^{-\frac{2\ii}{\eps}\phi(y)}h_{1}\left(\frac{2}{\eps}(\phi(y)-\phi(\xi))\right)\,\dd y\nonumber\\
	%&=(\ii\eps)^{3}\int_{\xi}^{\eta}b(y)b_{2}(y)h_{1}\left(\frac{2}{\eps}(\phi(\xi)-\phi(y))\right)\,\dd y\nonumber\\
	%&=(\ii\eps)^{4}\int_{\xi}^{\eta}b_{0}(y)b_{2}(y)\frac{\dd}{\dd y}\left[h_{2}\left(\frac{2}{\eps}(\phi(\xi)-\phi(y))\right)\right]\,\dd y\nonumber\\
	%&=(\ii\eps)^{4}\Bigg\{b_{0}(\eta)b_{2}(\eta)h_{2}\left(\frac{2}{\eps}(\phi(\xi)-\phi(\eta))\right)-\int_{\xi}^{\eta}\left(b_{0}(y)b_{2}(y)\right)^{\prime}h_{2}\left(\frac{2}{\eps}(\phi(\xi)-\phi(y))\right)\,\dd y\Bigg\}\nonumber\\
	&=(\ii\eps)^{4}\Bigg\{b_{0}(\eta)b_{2}(\eta)h_{2}\left(-\frac{2}{\eps}s_{n}\right)\nonumber\\
	&\quad-(\ii\eps)\left(b_{1}(\eta)b_{2}(\eta)+b_{0}(\eta)b_{3}(\eta)\right)h_{3}\left(-\frac{2}{\eps}s_{n}\right)\Bigg\}+\mathcal{O}_{\varepsilon,h}(\eps^{2}h^{3}\min(\eps,h)).
\end{align}
In order to approximate $J_{6}$ we use the identity
\begin{align}\label{help2}
	\ee^{\frac{2\ii}{\eps}\phi(\xi)}&\ee^{-\frac{2\ii}{\eps}\phi(y)}h_{2}\left(\frac{2}{\eps}(\phi(y)-\phi(\xi))\right)\nonumber\\
	&=-h_{2}\left(\frac{2}{\eps}(\phi(\xi)-\phi(y))\right)+\frac{2\ii}{\eps}(\phi(\xi)-\phi(y))h_{1}\left(\frac{2}{\eps}(\phi(\xi)-\phi(y))\right)\Comma\
\end{align}
which allows us to split the integral into $J_{6}=\widetilde{J}_{6}+\widehat{J}_{6}$, with
\begin{align}
	\widetilde{J}_{6}&:=(\ii\eps)^{4}\int_{\xi}^{\eta}b(y)b_{3}(y)h_{2}\left(\frac{2}{\eps}(\phi(\xi)-\phi(y))\right)\,\dd y\Comma\nonumber\\
	\widehat{J}_{6}&:=2(\ii\eps)^{3}\int_{\xi}^{\eta}b(y)b_{3}(y)\left(\phi(\xi)-\phi(y)\right)h_{1}\left(\frac{2}{\eps}(\phi(\xi)-\phi(y))\right)\,\dd y\period\nonumber
\end{align}
One SAM-step for $\widetilde{J}_{6}$ leads to
\begin{align}\label{J6tilde}
	\widetilde{J}_{6}&=(\ii\eps)^{5}b_{0}(\eta)b_{3}(\eta)h_{3}\left(-\frac{2}{\eps}s_{n}\right)+\mathcal{O}_{\varepsilon,h}(\eps^{2}h^{3}\min(\eps,h)) \period
\end{align}
For the approximation of $\widehat{J}_{6}$ we start with one SAM-step to see that
\begin{align}\label{J6hat_1}
	\widehat{J}_{6}
	%&=2(\ii\eps)^{4}\int_{\xi}^{\eta}b_{0}(y)b_{3}(y)\left(\phi(\xi)-\phi(y)\right)\frac{\dd}{\dd y}\left[h_{2}\left(\frac{2}{\eps}(\phi(\xi)-\phi(y))\right)\right]\,\dd y\nonumber\\
	&=2(\ii\eps)^{4}\Bigg\{b_{0}(\eta)b_{3}(\eta)(-s_{n})h_{2}\left(-\frac{2}{\eps}s_{n}\right)\nonumber\\
	&\quad-\underbrace{\int_{\xi}^{\eta}\left[b_{0}(y)b_{3}(y)\left(\phi(\xi)-\phi(y)\right)\right]^{\prime}h_{2}\left(\frac{2}{\eps}(\phi(\xi)-\phi(y))\right)\,\dd y}_{=:I}\Bigg\}\period
\end{align}
Now, using $b_{3}(y)\phi^{\prime}(y)=\frac{1}{2}b_{2}^{\prime}(y)$, we split up the integral $I$ as
\begin{align}\label{I}
	I&=\int_{\xi}^{\eta}\left(b_{0}^{\prime}(y)b_{3}(y)+b_{0}(y)b_{3}^{\prime}(y)\right)\left(\phi(\xi)-\phi(y)\right)h_{2}\left(\frac{2}{\eps}(\phi(\xi)-\phi(y))\right)\,\dd y\nonumber\\
	&\quad-\frac{1}{2}\int_{\xi}^{\eta}b_{0}(y)b_{2}^{\prime}(y)h_{2}\left(\frac{2}{\eps}(\phi(\xi)-\phi(y))\right)\,\dd y\nonumber\\
	&=(\ii\eps)\left(b_{1}(\eta)b_{3}(\eta)+b_{0}(\eta)b_{4}(\eta)\right)(-s_{n})h_{3}\left(-\frac{2}{\eps}s_{n}\right)\nonumber\\
	&\quad-(\ii\eps)\int_{\xi}^{\eta}\left[\left(b_{1}(y)b_{3}(y)+b_{0}(y)b_{4}(y)\right)\left(\phi(\xi)-\phi(y)\right)\right]^{\prime}h_{3}\left(\frac{2}{\eps}(\phi(\xi)-\phi(y))\right)\,\dd y\nonumber\\
	&\quad-\frac{1}{2}(\ii\eps)b_{0}(\eta)b_{3}(\eta)h_{3}\left(-\frac{2}{\eps}s_{n}\right)+\mathcal{O}_{\varepsilon,h}(\varepsilon^{2}\min(h^{4}\eps^{-4},h^{3}\eps^{-3}))\Comma
\end{align}
where we used one SAM-step for each integral in the last equation. Note that, combined with the $\mathcal{O}(\varepsilon^{4})$-factor from (\ref{J6hat_1}), the first term in (\ref{I}) is already of the desired order $\mathcal{O}_{\varepsilon,h}(\varepsilon^{2}h^{3}\min(\varepsilon,h))$ and can thus be omitted. Using (\ref{h_p_order}) we see that the second term in (\ref{I}) has the same order, and hence can also be neglected. Thus, we have
\begin{align}\label{J6hat_2}
	\widehat{J}_{6}&=2(\ii\eps)^{4}\Bigg\{b_{0}(\eta)b_{3}(\eta)(-s_{n})h_{2}\left(-\frac{2}{\eps}s_{n}\right)+\frac{1}{2}(\ii\eps)b_{0}(\eta)b_{3}(\eta)h_{3}\left(-\frac{2}{\eps}s_{n}\right)\Bigg\}\nonumber\\
    &\quad+\mathcal{O}_{\varepsilon,h}(\eps^{2}h^{3}\min(\eps,h))\Comma
\end{align}
and hence, by combining (\ref{J6tilde}) and (\ref{J6hat_2}),
\begin{align}\label{J6}
	J_{6}
	%&=\widetilde{J}_{6}+\widehat{J}_{6}\nonumber\\
	&=2(\ii\eps)^{4}b_{0}(\eta)b_{3}(\eta)(-s_{n})h_{2}\left(-\frac{2}{\eps}s_{n}\right)+2(\ii\eps)^{5}b_{0}(\eta)b_{3}(\eta)h_{3}\left(-\frac{2}{\eps}s_{n}\right)\nonumber\\
    &\quad+\mathcal{O}_{\varepsilon,h}(\eps^{2}h^{3}\min(\eps,h)) \period
\end{align}
Finally, we summarize the approximations (\ref{m_2_approximation})-(\ref{J4}), (\ref{J5}), and (\ref{J6}) in the following lemma.
\begin{Lemma}\label{lemma_m2}
Let Hypothesis \ref{hypothesis_A} be satisfied and define
\begin{align}\label{Q_2}
	Q_{2}&(\eta,\xi):=\nonumber\\
    &-\ii\eps Q_{S}[bb_{0}](\eta,\xi)\nonumber\\
	&-\eps^{2}\left[b_{0}(\xi)b_{0}(\eta)h_{0}\left(-\frac{2}{\eps}s_{n}\right)-b_{0}(\xi)^{2}-Q_{S}[bb_{1}](\eta,\xi)\right]\nonumber\\
	&+\ii\eps^{3}\left[b_{0}(\xi)b_{1}(\eta)-b_{1}(\xi)b_{0}(\eta)\right]h_{1}\left(-\frac{2}{\eps}s_{n}\right)\nonumber\\
	&+\eps^{4}\left[\left(b_{0}(\xi)+b_{0}(\eta)\right)b_{2}(\eta)-b_{1}(\xi)b_{1}(\eta)-2b_{0}(\eta)b_{3}(\eta)s_{n}\right]h_{2}\left(-\frac{2}{\eps}s_{n}\right)\nonumber\\
	&+\ii\eps^{5}\left[\left(b_{0}(\eta)-b_{0}(\xi)\right)b_{3}(\eta)-\left(b_{1}(\eta)-b_{1}(\xi)\right)b_{2}(\eta)\right]h_{3}\left(-\frac{2}{\eps}s_{n}\right)
\end{align}
and
\begin{align}\label{Q2}
	\mathbf{Q}_{2}(\eta,\xi):=
	\begin{pmatrix}
		Q_{2}(\eta,\xi) & 0 \\
		0 & \overline{Q_{2}(\eta,\xi)} 
	\end{pmatrix}\period
\end{align}
Then there exists a constant $C\geq 0$ independent of $\varepsilon\in(0,\eps_{0}]$, $h$, and $n$ such that
\begin{align}\label{Q2_estimate}
	\lVert \mathbf{M}_{2}^{\eps}(\eta,\xi)-\mathbf{Q}_{2}(\eta,\xi)\rVert_{\infty} \leq C\eps h^{4}\max(\eps,h)\period
\end{align}
\end{Lemma}
\begin{proof}
    Since $Q_{2}$ is defined precisely through the approximations (\ref{m_2_approximation})-(\ref{J4}), (\ref{J5}), and (\ref{J6}) by neglecting the $\mathcal{O}_{\eps,h}(\cdot)$-terms, estimate (\ref{Q2_estimate}) follows by noticing that $\mathcal{O}_{\varepsilon,h}(\eps^{2}h^{3}\min(\eps,h))+\mathcal{O}_{\varepsilon,h}(\eps h^{5})+\mathcal{O}_{\varepsilon,h}(\eps^{2} h^{5})=\mathcal{O}_{\varepsilon,h}(\eps h^{4}\max(\eps,h))$.
\end{proof}

\subsection{Approximation of the matrix \texorpdfstring{$\mathbf{M}_{3}^{\eps}$}{M3}}\label{M3}
%\subsection{\texorpdfstring{$\mathbf{M}_{3}^{\eps}$}{M\textsubscript{3}\textsuperscript{3}}}
%
Finally, in this subsection, we construct an approximation for the integral
\begin{align}\label{M3_matrix}
    \mathbf{M}_{3}^{\eps}(\eta;\xi)=\int_{\xi}^{\eta}\mathbf{N}^{\eps}(y)\mathbf{M}_{2}^{\eps}(y;\xi)\,\mathrm{d}y\period
\end{align}
While the desired error order for this approximation would be $\mathcal{O}_{\varepsilon,h}(\eps h^{3}\min(\eps,h))$ (see the discussion after (\ref{picard_remainder})), this will again not be possible here -- like in Subsection~\ref{M2}. Since the matrix $\mathbf{M}_{2}^{\eps}$ is diagonal, $\mathbf{M}_{3}^{\eps}$ is off-diagonal. We will now study $m_{3}^{\eps}(\eta,\xi):=\left(\mathbf{M}_{3}^{\eps}(\eta,\xi)\right)_{2,1}$ as the entry $\left(\mathbf{M}_{3}^{\eps}(\eta,\xi)\right)_{1,2}$ is just its complex conjugate.
To approximate $m_{3}^{\eps}$, we have to insert an approximation for $\mathbf{M}_{2}^{\eps}$ in the expression for $m_{3}^{\eps}$ and could of course simply use (\ref{Q2}). However, we will instead insert a weaker approximation as this will not result in a reduced error order for the resulting scheme and, at the same time, it leads to a shorter quadrature formula for $m_{3}^{\eps}$. The weaker approximation for $\mathbf{M}^{\eps}_{2}$ can be derived by using $\mathbf{Q}_{1}^{1,1}$ (see (\ref{Q1})) to approximate $\mathbf{M}_{1}^{\eps}$ and this was already used in \cite[Eq.\ (2.29)-(2.30)]{Arnold2011WKBBasedSF}:
\begin{align}\label{m2_weak}
	\left(\mathbf{M}_{2}^{\eps}(\eta;\xi)\right)_{1,1}=m_{2}^{\eps}(\eta,\xi)
	%&=\int_{\xi}^{\eta}b(y)\ee^{-\frac{2\ii}{\eps}\phi(y)}\left\{Q_{1}^{1}(y,\xi)+\mathcal{O}(\eps h\min(\eps,h))\right\}\,\dd y\nonumber\\
	%=I_{1}+I_{2}+I_{3}+\mathcal{O}_{\varepsilon,h}(\eps h^{2}\min(\eps,h))
    &=-(\ii\eps)\frac{\eta-\xi}{2}\left[b(\eta)b_{0}(\eta)+b(\xi)b_{0}(\xi)\right]\nonumber\\
    &\quad+(\ii\eps)^{2}b_{0}(\xi)b_{0}(\eta)h_{1}\left(\frac{2}{\eps}(\phi(\xi)-\phi(\eta))\right)\nonumber\\
    &\quad+(\ii\eps)^{3}b_{1}(\eta)\left[b_{0}(\eta)-b_{0}(\xi)\right]h_{2}\left(\frac{2}{\eps}(\phi(\xi)-\phi(\eta))\right)\nonumber\\
    &\quad+\mathcal{O}_{\varepsilon,h}(\eps h^{3})
\end{align}
%
% with
% %
% \begin{align}
% 	I_{1}&=-(\ii\eps)\frac{\eta-\xi}{2}\left[b(\eta)b_{0}(\eta)+b(\xi)b_{0}(\xi)\right]+\mathcal{O}_{\varepsilon,h}(\eps h^{3})\Comma\\
% 	I_{2}&=(\ii\eps)^{2}b_{0}(\xi)\left[b_{0}(\eta)h_{1}\left(\frac{2}{\eps}(\phi(\xi)-\phi(\eta))\right)-(\ii\eps)b_{1}(\eta)h_{2}\left(\frac{2}{\eps}(\phi(\xi)-\phi(\eta))\right)\right]\nonumber\\
% 	&\quad+\mathcal{O}_{\varepsilon,h}(\eps h^{2}\min(\eps,h))\Comma\\
% 	I_{3}&=(\ii\eps)^{3}b_{0}(\eta)b_{1}(\eta)h_{2}\left(\frac{2}{\eps}(\phi(\xi)-\phi(\eta))\right)+\mathcal{O}_{\varepsilon,h}(\eps h^{2}\min(\eps,h))\period
% \end{align}
%
Inserting the approximation (\ref{m2_weak}) of $m_{2}^{\eps}(y,\xi)$ for $\left(\mathbf{M}_{2}^{\eps}(y;\xi)\right)_{1,1}$ thus yields
\begin{align}
	m_{3}^{\eps}(\eta,\xi)&=\int_{\xi}^{\eta}\left(\mathbf{N}^{\eps}(y)\right)_{2,1}\left(\mathbf{M}_{2}^{\eps}(y;\xi)\right)_{1,1}\,\dd y\nonumber\\
	%&=\int_{\xi}^{\eta}b(y)\ee^{\frac{2\ii}{\eps}\phi(y)}\Bigg\{-(\ii\eps)\frac{y-\xi}{2}\left[b(y)b_{0}(y)+b(\xi)b_{0}(\xi)\right]\nonumber\\
	%&\quad+(\ii\eps)^{2}b_{0}(\xi)b_{0}(y)h_{1}\left(\frac{2}{\eps}(\phi(\xi)-\phi(y))\right)\nonumber\\
	%&\quad+(\ii\eps)^{3}b_{1}(y)\left[b_{0}(y)-b_{0}(\xi)\right]h_{2}\left(\frac{2}{\eps}(\phi(\xi)-\phi(y))\right)+\mathcal{O}_{\varepsilon,h}(\eps h^{3})\Bigg\}\,\dd y\nonumber\\
	%&
    &=K_{1}+K_{2}+K_{3}+K_{4}+K_{5}+\mathcal{O}_{\varepsilon,h}(\eps h^{4})\Comma
\end{align}
where
\begin{align}
	K_{1}&:=-\frac{(\ii\eps)}{2}\int_{\xi}^{\eta}b(y)^{2}b_{0}(y)(y-\xi)\ee^{\frac{2\ii}{\eps}\phi(y)}\,\dd y\Comma\nonumber\\
	K_{2}&:=-\frac{(\ii\eps)}{2}b(\xi)b_{0}(\xi)\int_{\xi}^{\eta}b(y)(y-\xi)\ee^{\frac{2\ii}{\eps}\phi(y)}\,\dd y\Comma\nonumber\\
	K_{3}&:=(\ii\eps)^{2}b_{0}(\xi)\int_{\xi}^{\eta}b(y)b_{0}(y)\ee^{\frac{2\ii}{\eps}\phi(y)}h_{1}\left(\frac{2}{\eps}(\phi(\xi)-\phi(y))\right)\,\dd y\Comma\nonumber\\
	K_{4}&:=-(\ii\eps)^{3}b_{0}(\xi)\int_{\xi}^{\eta}b(y)b_{1}(y)\ee^{\frac{2\ii}{\eps}\phi(y)}h_{2}\left(\frac{2}{\eps}(\phi(\xi)-\phi(y))\right)\,\dd y\Comma\nonumber\\
	K_{5}&:=(\ii\eps)^{3}\int_{\xi}^{\eta}b(y)b_{0}(y)b_{1}(y)\ee^{\frac{2\ii}{\eps}\phi(y)}h_{2}\left(\frac{2}{\eps}(\phi(\xi)-\phi(y))\right)\,\dd y\period\nonumber
\end{align}
For the approximation of the integrals $K_{i}$, $i=1,\dots, 5$ it is convenient to introduce the abbreviations
\begin{align}
     c_{0}(y)&:=\frac{b(y)^{2}b_{0}(y)}{2\phi^{\prime}(y)}\Comma \quad c_{1}(y):=\frac{c_{0}^{\prime}(y)}{2\phi^{\prime}(y)}\Comma \quad d_{0}(y):=\frac{c_{0}(y)}{2\phi^{\prime}(y)}\Comma\label{c0_etc}\\
     d_{1}(y)&:=\frac{d_{0}^{\prime}(y)}{2\phi^{\prime}(y)}\Comma\quad e_{0}(y):=\frac{c_{1}(y)}{2\phi^{\prime}(y)}\Comma\quad f_{0}(y):=\frac{b_{0}(y)}{2\phi^{\prime}(y)}\Comma\label{d1_etc}\\
     f_{1}(y)&:=\frac{f_{0}^{\prime}(y)}{2\phi^{\prime}(y)}\Comma \quad g_{0}(y):=\frac{b_{1}(y)}{2\phi^{\prime}(y)}\Comma \quad \kappa(y):=b(y)b_{1}(y)\Comma\label{f1_etc}\\
     \kappa_{0}(y)&:=\frac{\kappa(y)}{2\phi^{\prime}(y)}\Comma\quad l(y):=b(y)b_{0}(y)b_{1}(y)\Comma\quad l_{0}(y):=\frac{l(y)}{2\phi^{\prime}(y)}\period\label{kappa0_etc}
\end{align}
Note that all of these functions are uniformly bounded w.r.t.\ $\eps$ (with $0<\eps\leq \eps_{0}$ for some sufficiently small $\eps_{0}$) on the interval $I$. Indeed, this is again guaranteed by Hypothesis~\ref{hypothesis_A}. 

Now, let us start with the approximation of $K_{1}$. By making three SAM-steps, we find
\begin{align}\label{K1}
	K_{1}%&=-\frac{(\ii\eps)}{2}\int_{\xi}^{\eta}c(y)(y-\xi)\ee^{\frac{2\ii}{\eps}\phi(y)}\,\dd y\nonumber\\
	%&=\frac{(\ii\eps)^{2}}{2}\ee^{\frac{2\ii}{\eps}\phi(\xi)}\int_{\xi}^{\eta}c_{0}(y)(y-\xi)\frac{\dd}{\dd y}\left[h_{1}\left(\frac{2}{\eps}(\phi(y)-\phi(\xi))\right)\right]\,\dd y\nonumber\\
	&=\frac{(\ii\eps)^{2}}{2}\ee^{\frac{2\ii}{\eps}\phi(\xi)}\Bigg\{c_{0}(\eta)(\eta-\xi)h_{1}\left(\frac{2}{\eps}s_{n}\right)\nonumber\\
	&\quad+(\ii\eps)\int_{\xi}^{\eta}\left(c_{1}(y)(y-\xi)+d_{0}(y)\right)\frac{\dd}{\dd y}\left[h_{2}\left(\frac{2}{\eps}(\phi(y)-\phi(\xi))\right)\right]\,\dd y\Bigg\}\nonumber\\
	&=\frac{(\ii\eps)^{2}}{2}\ee^{\frac{2\ii}{\eps}\phi(\xi)}\Bigg\{c_{0}(\eta)(\eta-\xi)h_{1}\left(\frac{2}{\eps}s_{n}\right)\nonumber\\
	&\quad+(\ii\eps)\left(c_{1}(\eta)(\eta-\xi)+d_{0}(\eta)\right)h_{2}\left(\frac{2}{\eps}s_{n}\right)\nonumber\\
	&\quad+(\ii\eps)^{2}\int_{\xi}^{\eta}\left(\frac{c_{1}'(y)}{2\phi'(y)}(y-\xi)+e_{0}(y)+d_{1}(y)\right)\frac{\dd}{\dd y}\left[h_{3}\left(\frac{2}{\eps}(\phi(y)-\phi(\xi))\right)\right]\,\dd y\Bigg\}\nonumber\\
	&=\frac{(\ii\eps)^{2}}{2}\ee^{\frac{2\ii}{\eps}\phi(\xi)}\Bigg\{c_{0}(\eta)(\eta-\xi)h_{1}\left(\frac{2}{\eps}s_{n}\right)\nonumber\\
	&\quad+(\ii\eps)\left(c_{1}(\eta)(\eta-\xi)+d_{0}(\eta)\right)h_{2}\left(\frac{2}{\eps}s_{n}\right)\nonumber\\
	&\quad+(\ii\eps)^{2}\left(e_{0}(\eta)+d_{1}(\eta)\right)h_{3}\left(\frac{2}{\eps}s_{n}\right)
	%&\quad-(\ii\eps)^{2}\int_{\xi}^{\eta}c_{2}^{\prime}(y)(y-\xi)h_{3}\left(\frac{2}{\eps}(\phi(y)-\phi(\xi))\right)\,\dd y\nonumber\\
	%&\quad-(\ii\eps)^{2}\int_{\xi}^{\eta}\left(c_{2}(y)+e_{0}^{\prime}(y)+d_{1}^{\prime}(y)\right)h_{3}\left(\frac{2}{\eps}(\phi(y)-\phi(\xi))\right)\,\dd y
    \Bigg\}
	+\mathcal{O}_{\varepsilon,h}(\eps h^{3}\min(\eps,h))\period
\end{align}
%
%Now, by omitting terms of magnitude $\mathcal{O}(\eps h^{3}\min(\eps,h))$ or higher we end up with 
%%
%\begin{align}
%	K_{1}&=\frac{(\ii\eps)^{2}}{2}\ee^{\frac{2\ii}{\eps}\phi(\xi)}\Bigg\{c_{0}(\eta)(\eta-\xi)h_{1}\left(\frac{2}{\eps}(\phi(\eta)-\phi(\xi))\right)\nonumber\\
%	&\quad+(\ii\eps)\left(c_{1}(\eta)(\eta-\xi)+d_{0}(\eta)\right)h_{2}\left(\frac{2}{\eps}(\phi(\eta)-\phi(\xi))\right)\nonumber\\
%	&\quad+(\ii\eps)^{2}\left(e_{0}(\eta)+d_{1}(\eta)\right)h_{3}\left(\frac{2}{\eps}(\phi(\eta)-\phi(\xi))\right)\Bigg\}+\mathcal{O}(\eps h^{3}\min(\eps,h))\period
%\end{align}
%
Next, the integral $K_{2}$ can be treated in the same way as $K_{1}$, i.e., by making three SAM-steps:
\begin{align}\label{K2}
	K_{2}&=\frac{(\ii\eps)^{2}}{2}b(\xi)b_{0}(\xi)\ee^{\frac{2\ii}{\eps}\phi(\xi)}\Bigg\{b_{0}(\eta)(\eta-\xi)h_{1}\left(\frac{2}{\eps}s_{n}\right)\nonumber\\
	&\quad+(\ii\eps)\left(b_{1}(\eta)(\eta-\xi)+f_{0}(\eta)\right)h_{2}\left(\frac{2}{\eps}s_{n}\right)\nonumber\\
	&\quad+(\ii\eps)^{2}\left(g_{0}(\eta)+f_{1}(\eta)\right)h_{3}\left(\frac{2}{\eps}s_{n}\right)\Bigg\}+\mathcal{O}_{\varepsilon,h}(\eps h^{3}\min(\eps,h))\period
\end{align}
For the approximation of $K_{3}$ we first make use of (\ref{help1}).  %$\ee^{\frac{2\ii}{\eps}\phi(y)}h_{1}\left(\frac{2}{\eps}(\phi(\xi)-\phi(y))\right)=-\ee^{\frac{2\ii}{\eps}\phi(\xi)}h_{1}\left(\frac{2}{\eps}(\phi(y)-\phi(\xi))\right)$
Then, we make two SAM-steps, yielding
\begin{align}\label{K3}
	K_{3}&=-(\ii\eps)^{2}b_{0}(\xi)\ee^{\frac{2\ii}{\eps}\phi(\xi)}\int_{\xi}^{\eta}b(y)b_{0}(y)h_{1}\left(\frac{2}{\eps}(\phi(y)-\phi(\xi))\right)\,\dd y\nonumber\\
	&=(\ii\eps)^{3}b_{0}(\xi)\ee^{\frac{2\ii}{\eps}\phi(\xi)}\Bigg\{b_{0}(\eta)^{2}h_{2}\left(\frac{2}{\eps}s_{n}\right)\nonumber\\
	&\quad+2(\ii\eps)b_{0}(\eta)b_{1}(\eta)h_{3}\left(\frac{2}{\eps}s_{n}\right)\Bigg\}+\mathcal{O}_{\varepsilon,h}(\eps h^{3}\min(\eps,h))\period
\end{align}
In order to construct an approximation of $K_{4}$, we use a similar splitting as for $J_{6}$ in Subsection~\ref{M2}. Indeed, applying the identity
\begin{align}\label{help3}
	\ee^{\frac{2\ii}{\eps}\phi(y)}h_{2}\left(\frac{2}{\eps}(\phi(\xi)-\phi(y))\right)&=-\ee^{\frac{2\ii}{\eps}\phi(\xi)}h_{2}\left(\frac{2}{\eps}(\phi(y)-\phi(\xi))\right)\nonumber\\&\quad+\ee^{\frac{2\ii}{\eps}\phi(\xi)}\frac{2\ii}{\eps}\left(\phi(y)-\phi(\xi)\right)h_{1}\left(\frac{2}{\eps}(\phi(y)-\phi(\xi))\right)\Comma
\end{align}
we can write $K_{4}=\widetilde{K}_{4}+\widehat{K}_{4}$, where
\begin{align}\label{K4tilde}
	\widetilde{K}_{4}&:=(\ii\eps)^{3}b_{0}(\xi)\ee^{\frac{2\ii}{\eps}\phi(\xi)}\int_{\xi}^{\eta}\kappa(y)h_{2}\left(\frac{2}{\eps}(\phi(y)-\phi(\xi))\right)\,\dd y\nonumber\\
	&=-(\ii\eps)^{4}b_{0}(\xi)\ee^{\frac{2\ii}{\eps}\phi(\xi)}\kappa_{0}(\eta)h_{3}\left(\frac{2}{\eps}s_{n}\right)+\mathcal{O}_{\varepsilon,h}\left(\eps h^{3}\min(\eps,h)\right)\Comma
\end{align}
and
\begin{align}\label{K4hat}
	\widehat{K}_{4}&:=2(\ii\eps)^{2}b_{0}(\xi)\ee^{\frac{2\ii}{\eps}\phi(\xi)}\int_{\xi}^{\eta}\kappa(y)\left(\phi(y)-\phi(\xi)\right)h_{1}\left(\frac{2}{\eps}(\phi(y)-\phi(\xi))\right)\,\dd y\nonumber\\
	%&=-2(\ii\eps)^{3}b_{0}(\xi)\ee^{\frac{2\ii}{\eps}\phi(\xi)}\Big\{\kappa_{0}(\eta)\left(\phi(\eta)-\phi(\xi)\right)h_{2}\left(\frac{2}{\eps}(\phi(\eta)-\phi(\xi))\right)\nonumber\\
	%&\quad-\int_{\xi}^{\eta}\left(\kappa_{0}^{\prime}(y)\left(\phi(y)-\phi(\xi)\right)+\kappa_{0}(y)\phi^{\prime}(y)\right)h_{2}\left(\frac{2}{\eps}(\phi(y)-\phi(\xi))\right)\,\dd y\Big\}\nonumber\\
	%&=-2(\ii\eps)^{3}b_{0}(\xi)\ee^{\frac{2\ii}{\eps}\phi(\xi)}\Big\{\kappa_{0}(\eta)\left(\phi(\eta)-\phi(\xi)\right)h_{2}\left(\frac{2}{\eps}(\phi(\eta)-\phi(\xi))\right)\nonumber\\
	%&\quad+(\ii\eps)\int_{\xi}^{\eta}\left(\kappa_{1}(y)\left(\phi(y)-\phi(\xi)\right)+\frac{1}{2}\kappa_{0}(y)\right)\frac{\dd}{\dd y}\left[h_{3}\left(\frac{2}{\eps}(\phi(y)-\phi(\xi))\right)\right]\,\dd y\Big\}\nonumber\\
	&=-2(\ii\eps)^{3}b_{0}(\xi)\ee^{\frac{2\ii}{\eps}\phi(\xi)}\Bigg\{\kappa_{0}(\eta)s_{n}h_{2}\left(\frac{2}{\eps}s_{n}\right)\nonumber\\
	&\quad+(\ii\eps)\frac{1}{2}\kappa_{0}(\eta)h_{3}\left(\frac{2}{\eps}s_{n}\right)\Bigg\}+\mathcal{O}_{\varepsilon,h}\left(\eps h^{3}\min(\eps,h)\right)\period
\end{align}
Here, we made one SAM-step for $\widetilde{K}_{4}$ and treated $\widehat{K}_{4}$ in the same way as $\widehat{J}_{6}$ in the previous subsection.
Thus, by combining (\ref{K4tilde}) and (\ref{K4hat}), we obtain
\begin{align}\label{K4}
	K_{4}&=-2(\ii\eps)^{3}b_{0}(\xi)\ee^{\frac{2\ii}{\eps}\phi(\xi)}\Bigg\{\kappa_{0}(\eta)s_{n}h_{2}\left(\frac{2}{\eps}s_{n}\right)+(\ii\eps)\kappa_{0}(\eta)h_{3}\left(\frac{2}{\eps}s_{n}\right)\Bigg\}\nonumber\\
    &\quad+\mathcal{O}_{\varepsilon,h}\left(\eps h^{3}\min(\eps,h)\right)\period
\end{align}
Finally, the approximation of $K_{5}$ can be derived in the same way as the one of $K_{4}$, by noticing that $l$ and $l_{0}$ now play the roles of $\kappa$ and $\kappa_{0}$, respectively. Therefore, it holds that
\begin{align}\label{K5}
		K_{5}&=2(\ii\eps)^{3}\ee^{\frac{2\ii}{\eps}\phi(\xi)}\Bigg\{l_{0}(\eta)s_{n}h_{2}\left(\frac{2}{\eps}s_{n}\right)+(\ii\eps)l_{0}(\eta)h_{3}\left(\frac{2}{\eps}s_{n}\right)\Bigg\}\nonumber\\
        &\quad+\mathcal{O}_{\varepsilon,h}\left(\eps h^{3}\min(\eps,h)\right)\period
\end{align}
We close this subsection with the following lemma, which summarizes the approximations (\ref{K1})-(\ref{K3}), (\ref{K4}), and (\ref{K5}).
\begin{Lemma}\label{lemma_m3}
	Let Hypothesis \ref{hypothesis_A} be satisfied and define
 %    \textcolor{blue}{
 %    \begin{align}\label{Q_3}
	% 	Q_{3}(\eta,\xi):&=-\eps^{2}\ee^{\frac{2\ii}{\eps}\phi(\xi)}\frac{\eta-\xi}{2}c_{0}(\eta)h_{1}\left(\frac{2}{\eps}s_{n}\right)\nonumber\\
	% 	&-\ii\eps^{3}\ee^{\frac{2\ii}{\eps}\phi(\xi)}\Bigg[\frac{1}{2}\left[c_{1}(\eta)(\eta-\xi)+d_{0}(\eta)\right]+b_{0}(\xi)b_{0}(\eta)^{2}+2s_{n}\kappa_{0}(\eta)\Bigg]h_{2}\left(\frac{2}{\eps}s_{n}\right)\nonumber\\
	% 	&+\eps^{4}\ee^{\frac{2\ii}{\eps}\phi(\xi)}\Bigg[\frac{1}{2}\left[e_{0}(\eta)+d_{1}(\eta)\right]+2b_{0}(\xi)b_{0}(\eta)b_{1}(\eta)+2\kappa_{0}(\eta)\Bigg]h_{3}\left(\frac{2}{\eps}s_{n}\right)
	% \end{align}
 %    }
	\begin{align}\label{Q_3}
		&Q_{3}(\eta,\xi):=-\eps^{2}\ee^{\frac{2\ii}{\eps}\phi(\xi)}\frac{\eta-\xi}{2}\left(c_{0}(\eta)+b(\xi)b_{0}(\xi)b_{0}(\eta)\right)h_{1}\left(\frac{2}{\eps}s_{n}\right)\nonumber\\
	  &-\ii\eps^{3}\ee^{\frac{2\ii}{\eps}\phi(\xi)}\Bigg[\frac{1}{2}\left[c_{1}(\eta)(\eta-\xi)+d_{0}(\eta)+b(\xi)b_{0}(\xi)\left(b_{1}(\eta)(\eta-\xi)+f_{0}(\eta)\right)\right]\nonumber\\
		&\quad\quad\quad\quad\quad\quad+b_{0}(\xi)b_{0}(\eta)^{2}+2s_{n}\left(l_{0}(\eta)-b_{0}(\xi)\kappa_{0}(\eta)\right)\Bigg]h_{2}\left(\frac{2}{\eps}s_{n}\right)\nonumber\\
		&+\eps^{4}\ee^{\frac{2\ii}{\eps}\phi(\xi)}\Bigg[\frac{1}{2}\left[e_{0}(\eta)+d_{1}(\eta)+b(\xi)b_{0}(\xi)\left(g_{0}(\eta)+f_{1}(\eta)\right)\right]\nonumber\\
		&\quad\quad\quad\quad\quad\;+2\left[b_{0}(\xi)b_{0}(\eta)b_{1}(\eta)+\left(l_{0}(\eta)-b_{0}(\xi)\kappa_{0}(\eta)\right)\right]\Bigg]h_{3}\left(\frac{2}{\eps}s_{n}\right)
	\end{align}
	and
	\begin{align}\label{Q3}
		\mathbf{Q}_{3}(\eta,\xi):=
		\begin{pmatrix}
			0 & \overline{Q_{3}(\eta,\xi)} \\ 
			Q_{3}(\eta,\xi) & 0
		\end{pmatrix}\period
	\end{align}
	Then there exists a constant $C\geq 0$ independent of $\varepsilon\in(0,\eps_{0}]$, $h$, and $n$ such that
	\begin{align}
		\lVert \mathbf{M}_{3}^{\eps}(\eta,\xi)-\mathbf{Q}_{3}(\eta,\xi)\rVert_{\infty} \leq C\eps h^{4}\period
	\end{align}
	
\end{Lemma}
\section{Numerical scheme and error analysis}\label{chap:scheme}
%Finally, we can define the matrices for the scheme. Recall that $\xi=x_{n}$ and $\eta=x_{n+1}$ and $S_{n}=\phi(x_{n+1})-\phi(x_{n})$.
%
Based on the Picard approximation (\ref{picard3}) as well as the quadratures (\ref{Q1}), (\ref{Q2}), and (\ref{Q3}) we can now define the numerical scheme: Let
\begin{align}\label{scheme_matrices}
    \mathbf{A}^{1}_{n}:=\eps\mathbf{Q}_{1}^{3,3}(x_{n+1},x_{n})\Comma\quad\mathbf{A}^{2}_{n}:=\eps^{2}\mathbf{Q}_{2}(x_{n+1},x_{n})\Comma\quad\mathbf{A}^{3}_{n}:=\eps^{3}\mathbf{Q}_{3}(x_{n+1},x_{n})\period
\end{align}
Given the initial value $Z_{0}$ we define
\begin{align}\label{scheme}
Z_{n+1}:=\left(\mathbf{I}+\mathbf{A}^{1}_{n}+\mathbf{A}^{2}_{n}+\mathbf{A}^{3}_{n}\right)Z_{n}\Comma\quad n=0,\dots,N-1\period
\end{align}
\anton{Due to the detailed expressions of the three above matrices, as given in Section~\ref{chap:construction_scheme}, we shall assume for this scheme that the function $a$ and its derivatives up to order 7 are explicitly available.}

The numerical solution of (\ref{Usystem}) can then be obtained through the inverse transform (\ref{U_inverse}):
\begin{align}\label{U_inverse_num}
    U_{n}:=\mathbf{P}^{-1}\exp\left(\frac{\ii}{\varepsilon}\mathbf{\Phi}^{\varepsilon}(x_{n})\right)Z_{n}\Comma\quad n=0,\dots,N\period
\end{align}
The method (\ref{scheme})-(\ref{U_inverse_num}) satisfies the following global error estimates:
\begin{Theorem}\label{thm_main}
Let Hypothesis \ref{hypothesis_A} be satisfied, \anton{and let the phase \eqref{phase} be computable analytically. } Let $Z$ and $U$ be the exact solutions of the IVPs (\ref{Zsystem}) and (\ref{Usystem}), respectively. Then, for $Z_{n}$ and $U_{n}$ being computed through the scheme (\ref{scheme})-(\ref{U_inverse_num}), there exists a generic constant $C\geq 0$ independent of $\eps\in(0,\eps_{0}]$, $h$, and $n$ such that
\begin{align}\label{global_error_Z}
	\lVert Z(x_{n})-Z_{n}\rVert_{\infty} %&\leq C \eps^{4}h^{2}\min(\eps,h)+C\eps^{3}h^{4}+C\eps^{4}h^{3}\nonumber\\
	&\leq  C \eps^{3}h^{3}\max(\eps,h)\Comma\quad n=0,\dots,N\Comma
\end{align}
\begin{align}\label{global_error_U}
	\lVert U(x_{n})-U_{n}\rVert_{\infty} %&\leq C\frac{h^{\gamma}}{\eps}+C \eps^{4}h^{2}\min(\eps,h)+C\eps^{3}h^{4}+C\eps^{4}h^{3}\nonumber\\
	\leq C\eps^{3}h^{3}\max(\eps,h)\Comma\quad n=0,\dots,N\period
\end{align}
\end{Theorem}
\begin{proof}
From the definitions (\ref{Q_1}), (\ref{Q_2}), and (\ref{Q_3}) it is evident that
\begin{align}
	\lVert \mathbf{A}^{1}_{n}\rVert_{\infty} &= \lVert \eps\mathbf{Q}_{1}^{3,3}(x_{n+1},x_{n})\rVert_{\infty} = \mathcal{O}_{\varepsilon,h}(\eps\min(\eps,h))\Comma\nonumber\\
	\lVert \mathbf{A}^{2}_{n}\rVert_{\infty} &= \lVert \eps^{2}\mathbf{Q}_{2}(x_{n+1},x_{n})\rVert_{\infty} = \mathcal{O}_{\varepsilon,h}(\eps^{3}h)\Comma\nonumber\\
	\lVert \mathbf{A}^{3}_{n}\rVert_{\infty} &= \lVert \eps^{3}\mathbf{Q}_{3}(x_{n+1},x_{n})\rVert_{\infty} = \mathcal{O}_{\varepsilon,h}(\eps^{4} h\min(\eps,h))\Comma
\end{align}
which implies that the one-step method (\ref{scheme}) is stable, with an $\varepsilon$-independent stability constant.
Further, due to (\ref{picard3}), the consistency error for $n=0,\dots,N-1$ reads
	\begin{align}\label{consistency_error}
		e_{n}&:=Z(x_{n+1})-\left(\mathbf{I}+\mathbf{A}^{1}_{n}+\mathbf{A}^{2}_{n}+\mathbf{A}^{3}_{n}\right)Z(x_{n})\nonumber\\
		&=Z(x_{n})+\left[\eps \mathbf{M}_{1}^{\eps}(x_{n+1};x_{n})+\eps^{2} \mathbf{M}_{2}^{\eps}(x_{n+1};x_{n})+\eps^{3} \mathbf{M}_{3}^{\eps}(x_{n+1};x_{n})\right]Z(x_{n})\nonumber\\
		&\quad-\left(\mathbf{I}+\mathbf{A}^{1}_{n}+\mathbf{A}^{2}_{n}+\mathbf{A}^{3}_{n}\right)Z(x_{n})+\mathcal{O}_{\eps,h}(\eps^{4}h^{3}\min(\eps,h))\nonumber\\
		&=\left[\eps\mathbf{M}_{1}^{\eps}(x_{n+1};x_{n})-\mathbf{A}^{1}_{n}+\eps^{2}\mathbf{M}_{2}^{\eps}(x_{n+1};x_{n})-\mathbf{A}^{2}_{n}+\eps^{3}\mathbf{M}_{3}^{\eps}(x_{n+1};x_{n})-\mathbf{A}^{3}_{n}\right]Z(x_{n})\nonumber\\
		&\quad+\mathcal{O}_{\eps,h}(\eps^{4}h^{3}\min(\eps,h))\Comma
	\end{align}
Now, Lemma~\ref{lemma_m1}, Lemma~\ref{lemma_m2}, and Lemma~\ref{lemma_m3} imply
\begin{align}
	\eps\mathbf{M}_{1}^{\eps}(x_{n+1};x_{n})-\mathbf{A}^{1}_{n}&=\mathcal{O}_{\varepsilon,h}(\eps^{4}h^{3}\min(\eps,h))\Comma\nonumber\\
	\eps^{2}\mathbf{M}_{2}^{\eps}(x_{n+1};x_{n})-\mathbf{A}^{2}_{n}&=\mathcal{O}_{\varepsilon,h}(\eps^{3}h^{4}\max(\eps,h))\Comma\nonumber\\
	\eps^{3}\mathbf{M}_{3}^{\eps}(x_{n+1};x_{n})-\mathbf{A}^{3}_{n}&=\mathcal{O}_{\eps,h}(\eps^{4}h^{4})\Comma
\end{align}
which, together with (\ref{consistency_error}), yields
\begin{align}
	e_{n}=\mathcal{O}_{\varepsilon,h}(\eps^{3}h^{4}\max(\eps,h))\period
\end{align}
Hence, the one-step method (\ref{scheme}) is consistent and therefore convergent with the global error estimate (\ref{global_error_Z}). Estimate (\ref{global_error_U}) now follows from (\ref{global_error_Z}) and the inverse transformation (\ref{U_inverse}), using the unitarity of the matrices $\mathbf{P}^{-1}$ and $\exp\left(\frac{\ii}{\eps}\mathbf{\Phi}^{\eps}\right)$. This proves the claim.
\end{proof}
\begin{Remark}\label{remark_h4}
	We emphasize that, in practical applications, when solving IVP (\ref{schroedinger_IVP}) using a WKB-based method like the one presented above, one may want to use a grid size $h \gg\eps$. In such scenarios the global error of the third order scheme (\ref{scheme}) is even proportional to $h^{4}$, according to the estimates (\ref{global_error_Z}) and (\ref{global_error_U}). This property is observable in Figure~\ref{plot:airy_err_vs_h_vpa} of the next section. 
\end{Remark}
%
%\begin{Corollary}[\cite{ABN11}]
%	Let Hypothesis \ref{hypothesis_A} be satisfied. If $U(x)$ is the exact solution of the IVP (\ref{Usystem}) and we denote by $U_{n}$ the numerical approximation of $U(x_{n})$, defined by the numerical approximation $Z_{n}$ from (\ref{scheme}) and the inverse transform (\ref{U_inverse}), then there exists a constant $C\geq 0$ independent of $\eps$ and $h$, such that
%	\begin{align}%\label{global_error_U}
%		\lVert U(x_{n})-U_{n}\rVert %&\leq C\frac{h^{\gamma}}{\eps}+C \eps^{4}h^{2}\min(\eps,h)+C\eps^{3}h^{4}+C\eps^{4}h^{3}\nonumber\\
%		&\leq C\frac{h^{\gamma}}{\eps} + \eps^{3}h^{3}\max(\eps,h)\Comma\quad n=0,\dots,N\period
%	\end{align}
%	Here, $\gamma$ is the numerical order of the quadrature rule to compute the phase (\ref{phase}). In particular, the first term in (\ref{global_error_U}) drops out completely if the phase is computed exactly.
%\end{Corollary}
%
\subsection{Refined error estimate incorporating phase errors}\label{sec:numphase}
\anton{In %\sout{the framework of}
Theorem~\ref{thm_main} we %\sout{implicitly} 
assumed} that the phase (\ref{phase}) is exactly available. Indeed, this is the case in several relevant applications, e.g., for so-called RTD-models (e.g., see \cite{Mennemann2013TransientSS,Sun1998ResonantTD}), where the coefficient function $a$ is typically piecewise linear. Thus, in such scenarios the method (\ref{scheme})-(\ref{U_inverse_num}) is asymptotically correct w.r.t.\ $\eps$.

In general, however, the integral within the definition of the phase (\ref{phase}) cannot be expected to be exactly computable. In \cite{Arnold2022WKBmethodFT}, the authors extended the error analysis of the numerical methods from \cite{Arnold2011WKBBasedSF} to the case of a numerically computed phase. We emphasize that, with the exact same strategy, the estimates (\ref{global_error_Z}) and (\ref{global_error_U}) can be generalized. To this end, let us first collect the basic assumptions from \cite{Arnold2022WKBmethodFT}.

Firstly, for an approximate phase $\tilde{\phi}\approx \phi$ we write $\tilde{\phi}(x)=\tilde{\phi}_{1}(x)-\eps^{2}\tilde{\phi}_{2}(x)$, with $\tilde{\phi}_{1}(x)$ and $\tilde{\phi}_{2}(x)$ being numerical approximations to $\int_{x_{0}}^{x}\sqrt{a(y)}\,\mathrm{d}y$ and $\int_{x_{0}}^{x}b(y)\,\mathrm{d}y$, respectively. Then we impose the following assumption:
\begin{Hypothesis}\label{hypothesis_B}
	Let $\tilde{\phi}_{1}, \tilde{\phi}_{2}\in C^{6}\left(I\right)$ and $\tilde{\phi}^{\prime}_{1}\geq \widetilde{C} > 0$.
\end{Hypothesis}
In particular, Hypothesis~\ref{hypothesis_B} implies that there are positive constants $E$, $E^{\prime}$ and $E^{\prime\prime}$ such that the approximate phase $\tilde{\phi}$ satisfies the following $L^{\infty}(I)$-error bounds uniformly in $\varepsilon\in (0,\tilde{\eps}_{0}]$:
\begin{align}
	\lVert \phi-\tilde{\phi}\rVert_{L^{\infty}(I)}\leq E\Comma\quad\label{Es}\lVert \phi^{\prime}-\tilde{\phi}^{\prime}\rVert_{L^{\infty}(I)}\leq E^{\prime}\Comma\quad \lVert \phi^{\prime\prime}-\tilde{\phi}^{\prime\prime}\rVert_{L^{\infty}(I)}\leq E^{\prime\prime}\Comma%\\
	%\lVert \phi^{\prime}-\tilde{\phi}^{\prime}\rVert_{L^{\infty}(I)}&\leq E^{\prime}\Comma\label{E'}\\
	%\lVert \phi^{\prime\prime}-\tilde{\phi}^{\prime\prime}\rVert_{L^{\infty}(I)}&\leq E^{\prime\prime}\label{E''}
\end{align}
where $\tilde{\eps}_{0}\leq \eps_{0}$ is sufficiently small. Note that the functions $a$, $b$, and $b_{p}$, $p=0,\dots,5$, as well as all functions defined in (\ref{c0_etc})-(\ref{kappa0_etc}) can be computed exactly, since $\phi^{\prime}$ is explicitly known. However, since we assume here an approximate phase $\tilde{\phi}$, it appears more consistent to use instead the numerical approximations $\tilde{a}$, $\tilde{b}$, $\tilde{b}_{p}$ (for $p=0,\dots,5$), $\tilde{c}_{0}$, $\tilde{c}_{1}$, $\tilde{d}_{0}$, $\tilde{d}_{1}$, $\tilde{e}_{0}$, $\tilde{f}_{0}$, $\tilde{f}_{1}$, $\tilde{g}_{0}$, $\tilde{\kappa}_{0}$, and $\tilde{l}_{0}$, which can be obtained using exact derivatives of $\tilde{\phi}$ (e.g., when using the spectral integration from \cite[§4]{Arnold2022WKBmethodFT} to obtain $\tilde{\phi}$, the exact derivatives are readily available). Together with the approximate phase increments $\tilde{s}_{n}:=\tilde{\phi}(x_{n+1})-\tilde{\phi}(x_{n})$ this leads to the approximate matrices $\tilde{\mathbf{A}}_{n}^{1}\approx \mathbf{A}_{n}^{1}$, $\tilde{\mathbf{A}}_{n}^{2}\approx \mathbf{A}_{n}^{2}$ and $\tilde{\mathbf{A}}_{n}^{3}\approx \mathbf{A}_{n}^{3}$ which are defined in analogy to (\ref{scheme_matrices}). % (corresponding to $\mathbf{A}_{n}^{1}$, $\mathbf{A}_{n}^{2}$ and $\mathbf{A}_{n}^{3}$).
The third order method involving the approximate phase then reads
\begin{align}\label{scheme_perturbed}
	\tilde{Z}_{n+1}:=\left(\mathbf{I}+\tilde{\mathbf{A}}^{1}_{n}+\tilde{\mathbf{A}}^{2}_{n}+\tilde{\mathbf{A}}^{3}_{n}\right)\tilde{Z}_{n}\Comma\quad n=0,\dots,N-1\Comma
\end{align}
with $\tilde{Z}_{0}:=Z_{0}$ being the initial value. The corresponding (perturbed) inverse transform to the $U$-variable is then given by (see (\ref{U_inverse}))
\begin{align}\label{U_inverse_perturbed}
	\tilde{U}_{n}=\mathbf{P}^{-1}\exp\left(\frac{\ii}{\eps}\tilde{\mathbf{\Phi}}^{\eps}(x_{n})\right)\tilde{Z}_{n}\Comma\quad n=0,\dots,N\Comma
\end{align}
with $\tilde{\mathbf{\Phi}}^{\eps}:=\operatorname{diag}(\tilde{\phi},-\tilde{\phi})$ being the perturbed phase matrix. 

Now, given the above considerations, we stress that, by using the exact same arguments as in the proof of \cite[Theorem 3.2]{Arnold2022WKBmethodFT}, we can extend the estimates (\ref{global_error_Z})-(\ref{global_error_U}) to incorporate also the phase errors:
\begin{Theorem}[Adaption of Thm.\ 3.2 from \cite{Arnold2022WKBmethodFT}]\label{thm_main_perturbed}
	Let Hypothesis~\ref{hypothesis_A} be satisfied and let the approximate phase $\tilde{\phi}$ satisfy Hypothesis~\ref{hypothesis_B}. Let $Z$ and $U$ be the exact solutions of the IVPs (\ref{Zsystem}) and (\ref{Usystem}), respectively. Then, for $\tilde{Z}_{n}$ and $\tilde{U}_{n}$ being computed through the scheme (\ref{scheme_perturbed})-(\ref{U_inverse_perturbed}), there exists a generic constant $C\geq 0$ independent of $\eps\in(0,\tilde{\eps}_{0}]$, $h$, and $n$ such that
	\begin{align}\label{global_error_Z_perturbed}
		\lVert Z(x_{n})-\tilde{Z}_{n}\rVert &\leq C\eps^{3}h^{3}\max(\eps,h)\nonumber\\
		&\quad+C\eps\left[\min(\eps,E)+\eps(E^{\prime}+E^{\prime\prime})\right]\Comma\quad n=0,\dots,N\Comma
	\end{align}
	\begin{align}\label{global_error_U_perturbed}
		\lVert U(x_{n})-\tilde{U}_{n}\rVert &\leq C \frac{E}{\eps} + C\eps^{3}h^{3}\max(\eps,h)\nonumber\\
		&\quad+ C\eps\left[\min(\eps,E)+\eps(E^{\prime}+E^{\prime\prime})\right]\Comma\quad n=0,\dots,N\period
	\end{align}
Here, the constants $E,E^{\prime}$, and $E^{\prime\prime}$ are from (\ref{Es}).
\end{Theorem}
Let us compare this result with the estimates (\ref{global_error_Z}) and (\ref{global_error_U}): The new (additional) second term in (\ref{global_error_Z_perturbed}) is caused by the impact of the approximate phase during the computation of the step update (\ref{scheme_perturbed}), whereas the new first term in (\ref{global_error_U_perturbed}) is due to the inverse transform (\ref{U_inverse_perturbed}) involving the perturbed phase matrix $\tilde{\mathbf{\Phi}}^{\eps}(x_{n})$. Obviously, the $\mathcal{O}(E/\varepsilon)$-term is rather unfavorable. %Since the computation of $\tilde{\phi}$ through a numerical quadrature of order $\gamma$ implies $E=\mathcal{O}(h^{\gamma})$, it is necessary to impose an upper step size limit $h\leq \bar{h}(\eps)=\eps^{\gamma}$, in order to recover the . Thus, one should aim for highly accurate quadrature to reduce the impact of the $\frac{1}{\eps}$-term in the estimate (\ref{global_error_U_perturbed}) as much as possible.
In order to reduce its impact as much as possible, one should hence aim for a highly accurate approximation $\tilde{\phi}$, e.g., by employing a spectral method. Indeed, such an approach has already proven useful in \cite{Arnold2022WKBmethodFT}, where the $\mathcal{O}(E/\varepsilon)$-term in their estimate was numerically invisible, since it was reduced below relative machine precision (compared to the other terms).
\subsection{Simplified third order scheme}\label{sec:simpl}
In this subsection, we present a simplified third order method for solving the IVP (\ref{Zsystem}). The basis for this method is the observation, that the inability to achieve the originally desired error order for the approximations of $\mathbf{M}_{2}^{\eps}$ and $\mathbf{M}_{3}^{\eps}$ in Section~\ref{chap:construction_scheme}, was due to the presence of non-oscillatory integrals in the computations. Indeed, it is evident from Subsection~\ref{M2} that the approximation of the non-oscillatory integral $J_{1}$ in (\ref{J1}) constrains the maximum achievable error order (w.r.t.\ $\eps$) for the approximation of $\mathbf{M}_{2}^{\eps}$ to $\mathcal{O}(\eps)$. Consequently, this implies that the maximum achievable error order (w.r.t.\ $\eps$) for any numerical scheme resulting from this approximation is $\mathcal{O}(\eps^{3})$. Given this unavoidable constraint, it seems unnecessary and superfluous to approximate the other terms with higher accuracy. Thus, we may want to weaken several approximations from Section~\ref{chap:construction_scheme} in order to derive a simplified scheme, which remains third order w.r.t.\ $h$, while still yielding the same asymptotic accuracy as the third order scheme (\ref{scheme})-(\ref{U_inverse_num}), i.e.\ $\mathcal{O}(\eps^{3})$ as $\eps\to 0$. This can be achieved as follows:

\medskip
\underline{Step 1 (approximation of $\mathbf{M}_{1}^{\eps}$ in (\ref{picard3})):} Instead of $\mathbf{Q}_{1}^{3,3}$, we shall use $\mathbf{Q}_{1}^{2,3}$ to approximate $\mathbf{M}_{1}^{\eps}$. This incurs an error of order $\mathcal{O}_{\eps,h}(\eps^{2}h^{3}\min(\eps,h))$ (instead of $\mathcal{O}_{\eps,h}(\eps^{3}h^{3}\min(\eps,h))$), see Lemma~\ref{lemma_m1}.

\underline{Step 2 (approximation of $\mathbf{M}_{2}^{\eps}$ in (\ref{picard3})):} To approximate $\mathbf{M}_{2}^{\eps}$, we shall use a simplified quadrature $\mathbf{Q}_{2,simp}$ (instead of $\mathbf{Q}_{2}$), which can be derived by inserting $\mathbf{Q}_{1}^{1,2}$ (instead of $\mathbf{Q}_{1}^{2,2}$) as an approximation for $\mathbf{M}_{1}^{\eps}$ in (\ref{M2_matrix}). The non-oscillatory integrals that then occur are again approximated applying Simpson's rule, and the oscillatory integrals are all approximated with an error order $\mathcal{O}_{\eps,h}(\eps h^{3}\min(\eps,h))$ (instead of $\mathcal{O}_{\eps,h}(\eps^{2} h^{3}\min(\eps,h))$), using only SAM-steps.

\underline{Step 3 (approximation of $\mathbf{M}_{3}^{\eps}$ in (\ref{picard3})):} To approximate $\mathbf{M}_{3}^{\eps}$, we shall use a simplified quadrature $\mathbf{Q}_{3,simp}$ (instead of $\mathbf{Q}_{3}$), which can be derived as follows: First, by inserting $\mathbf{Q}_{1}^{0,1}$ (instead of $\mathbf{Q}_{1}^{1,1}$) as an approximation for $\mathbf{M}_{1}^{\eps}$ in (\ref{M2_matrix}), we derive with only one SAM-step an approximation $\widetilde{\mathbf{Q}}_{2,simp}$ of $\mathbf{M}_{2}^{\eps}$, with an error order $\mathcal{O}_{\eps,h}(h^{2}\min(\eps,h))$. Then, by inserting $\widetilde{\mathbf{Q}}_{2,simp}$ in (\ref{M3_matrix}) one obtains an oscillatory integral that can be approximated using only SAM-steps with an error order $\mathcal{O}_{\eps,h}(h^{3}\min(\eps,h))$. This yields the quadrature $\mathbf{Q}_{3,simp}$.

\medskip
We stress that the implementation and analysis of Step~2 and Step~3 is fully straightforward and very similar to the computations in Subsection~\ref{M2} and Subsection~\ref{M3}, while Step~1 is trivial. The resulting quadratures $\mathbf{Q}_{2,simp}$ and $\mathbf{Q}_{3,simp}$, along with their corresponding error estimates, are provided by the following lemma.
\begin{Lemma}\label{lemma_simplified}
    Let Hypothesis \ref{hypothesis_A} be satisfied and define
\begin{align}
	Q_{2,simp}(\eta,\xi)&:=-\ii\eps Q_{S}[bb_{0}](\eta,\xi)
	-\eps^{2}b_{0}(\xi)b_{0}(\eta)h_{1}\left(-\frac{2}{\eps}s_{n}\right)\nonumber\\
	&-\ii\eps^{3}\left[b_{0}(\eta)\left(b_{1}(\eta)-2b_{2}(\eta)s_{n}\right)-b_{0}(\xi)b_{1}(\eta)\right]h_{2}\left(-\frac{2}{\eps}s_{n}\right)\nonumber\\
	&+\eps^{4}\left[b_{2}(\eta)\left(b_{0}(\xi)+b_{0}(\eta)\right)-b_{1}(\eta)^{2}\right]h_{3}\left(-\frac{2}{\eps}s_{n}\right)\Comma\label{Q_2_simp}\\
    Q_{3,simp}(\eta,\xi)&:=-2\ee^{\frac{2\ii}{\eps}\phi(\xi)}b_{0}(\eta)^{3}\left(\eps^{2}s_{n}h_{2}\left(\frac{2}{\eps}s_{n}\right)+\ii\eps^{3}h_{3}\left(\frac{2}{\eps}s_{n}\right)\right)\label{Q_3_simp}
\end{align}
and
\begin{align}
	\mathbf{Q}_{2,simp}(\eta,\xi)&:=
	\begin{pmatrix}
		Q_{2,simp}(\eta,\xi) & 0 \\
		0 & \overline{Q_{2,simp}(\eta,\xi)} 
	\end{pmatrix}\Comma\label{Q2_simp}\\
    \mathbf{Q}_{3,simp}(\eta,\xi)&:=
	\begin{pmatrix}
		0 & \overline{Q_{3,simp}(\eta,\xi)} \\
		Q_{3,simp}(\eta,\xi) & 0 
	\end{pmatrix}\period\label{Q3_simp}
\end{align}
Then there exists a generic constant $C\geq 0$ independent of $\varepsilon\in(0,\eps_{0}]$, $h$, and $n$ such that
\begin{align}
	\lVert \mathbf{M}_{2}^{\eps}(\eta,\xi)-\mathbf{Q}_{2,simp}(\eta,\xi)\rVert_{\infty} &\leq C\eps h^{4}\Comma\label{Q2_simp_estimate}\\
    \lVert \mathbf{M}_{3}^{\eps}(\eta,\xi)-\mathbf{Q}_{3,simp}(\eta,\xi)\rVert_{\infty} &\leq C h^{3}\min(\eps,h)\period\label{Q3_simp_estimate}
\end{align}
\end{Lemma}
The resulting simplified third order scheme then reads: Let $\mathbf{A}^{1,simp}_{n}:=\eps\mathbf{Q}_{1}^{2,3}(x_{n+1},x_{n})$, $\mathbf{A}^{2,simp}_{n}:=\eps^{2}\mathbf{Q}_{2,simp}(x_{n+1},x_{n})$, and $\mathbf{A}^{3,simpl}_{n}:=\eps^{3}\mathbf{Q}_{3,simp}(x_{n+1},x_{n})$. Given the initial value $Z_{0}$ we define
\begin{align}\label{scheme_simpl}
%Z_{n+1}^{simpl}:=\left(\mathbf{I}+\mathbf{A}^{1,simpl}_{n}+\mathbf{A}^{2,simpl}_{n}+\mathbf{A}^{3,simpl}_{n}\right)Z_{n}^{simpl}\Comma\quad n=0,\dots,N-1\period
Z_{n+1}:=\left(\mathbf{I}+\mathbf{A}^{1,simp}_{n}+\mathbf{A}^{2,simp}_{n}+\mathbf{A}^{3,simp}_{n}\right)Z_{n}\Comma\quad n=0,\dots,N-1\period
\end{align}
The simplified method (\ref{scheme_simpl}), (\ref{U_inverse_num}) satisfies the following global error estimate, which can be proven analogously to Theorem~\ref{thm_main}:
\begin{Theorem}\label{thm_simpl}
Let Hypothesis \ref{hypothesis_A} be satisfied. Let $Z$ and $U$ be the exact solutions of the IVPs (\ref{Zsystem}) and (\ref{Usystem}), respectively. Then, for $Z_{n}$ and $U_{n}$ being computed through the scheme (\ref{scheme_simpl}), (\ref{U_inverse_num}), there exists a generic constant $C\geq 0$ independent of $\eps\in(0,\eps_{0}]$, $h$, and $n$ such that
\begin{align}\label{global_error_Z_simp}
	\lVert Z(x_{n})-Z_{n}\rVert_{\infty}
	\leq  C \eps^{3}h^{3}\Comma\quad n=0,\dots,N\Comma
\end{align}
\begin{align}\label{global_error_U_simp}
	\lVert U(x_{n})-U_{n}\rVert_{\infty}
	\leq  C \eps^{3}h^{3}\Comma\quad n=0,\dots,N\period
\end{align}
\end{Theorem}
When comparing estimates (\ref{global_error_Z_simp})-(\ref{global_error_U_simp}) with estimates (\ref{global_error_Z})-(\ref{global_error_U}) for the scheme (\ref{scheme})-(\ref{U_inverse_num}), we observe that the downside of using the simplified scheme (\ref{scheme_simpl}), (\ref{U_inverse_num}) is the loss of the $\mathcal{O}(\max(\eps,h))$-factor. Indeed, while both methods have the same asymptotical (w.r.t.\ $\eps$) as well as numerical (w.r.t.\ $h$) order, this factor might still be beneficial, recalling that both parameters $\eps$ and $h$ are small. 
Further, we note that, in case of a numerically computed phase, Theorem~\ref{thm_simpl} can be generalized analogously to Theorem~\ref{thm_main_perturbed}. That is, by using the same strategy as in Subsection~\ref{sec:numphase}, we have the following error estimates for the simplified scheme including an approximated phase:
\begin{align}\label{global_error_Z_simp_numphase}
	\lVert Z(x_{n})-\tilde{Z}_{n}\rVert_{\infty}
	&\leq  C \eps^{3}h^{3}\nonumber\\
    &\quad+C\eps\left[\min(\eps,E)+\eps(E^{\prime}+E^{\prime\prime})\right]\Comma\quad n=0,\dots,N\Comma
\end{align}
\begin{align}\label{global_error_U_simp_numphase}
	\lVert U(x_{n})-\tilde{U}_{n}\rVert_{\infty}
	&\leq  C \frac{E}{\eps} + C\eps^{3}h^{3}\nonumber\\
	&\quad+ C\eps\left[\min(\eps,E)+\eps(E^{\prime}+E^{\prime\prime})\right]\Comma\quad n=0,\dots,N\period
\end{align}
\section{Numerical  results}\label{chap:simulations}
In this section we present and compare numerical results by applying the novel schemes from Section \ref{chap:scheme} to exemplary IVPs. Since a large part of the development of the two presented third order schemes (\ref{scheme})-(\ref{U_inverse_num}) and (\ref{scheme_simpl}), (\ref{U_inverse_num}) was based on the strategies from \cite{Arnold2011WKBBasedSF}, we shall also include their second order method into our comparison. This will also illustrate the efficiency gain of the new third order methods over the second order method from \cite{Arnold2011WKBBasedSF}.

For clarity of the presentation we shall in the following denote with \textit{WKB3} the third order scheme (\ref{scheme})-(\ref{U_inverse_num}), and with \textit{WKB3s} the simplified third order scheme (\ref{scheme_simpl}), (\ref{U_inverse_num}). Further, we denote with \textit{WKB2} the second order method from \cite{Arnold2011WKBBasedSF}, which yields global errors (w.r.t.\ the $U$-variable) of order $\mathcal{O}_{\eps,h}(\eps^{3}h^{2})$, provided that the phase (\ref{phase}) is explicitly available. We note that a single step of each of these three methods involves a different number of computational operations as well as function evaluations. Hence it is not clear a priori which of the three methods is the most efficient. We shall thus assess the overall efficiency of each method by comparing respective work-precision diagrams.% Regarding this, we note that every CPU time shown in this section is the average value of 5 measurements.

All computations here are carried out using \MATLAB version 23.2.0.2459199 (R2023b). Moreover, for the results of Figures~\ref{plot:airy_err_vs_h_vpa}-\ref{plot:airy_times_vs_err_vpa}, which involve very small approximation errors, we used the Advanpix Multiprecision Computing Toolbox for \MATLAB \cite{Advanpix} with quadruple-precision to avoid rounding errors. We note that this increases the computational times when compared to using the standard double-precision arithmetic of \MATLAB.
\subsection{First example: Airy equation}\label{sec:Ex1}
For the Airy equation, i.e., with $a(x)=x$ we consider the following IVP:
\begin{align}\label{eqn:Airy}
	\begin{cases}
		\varepsilon^{2}\varphi^{\prime\prime  }(x) + x \varphi(x) = 0 \Comma \quad x \in [1, 2] \Comma \\
		\varphi(1) = \Ai(-\frac{1}{\varepsilon^{2/3}}) + \ii \Bi(-\frac{1}{\varepsilon^{2/3}}) \Comma \\
		\varepsilon\varphi^{\prime}(1) = \anton{-} \varepsilon^{1/3}\left(\Ai^{\prime}(-\frac{1}{\varepsilon^{2/3}}) + \ii \Bi^{\prime}(-\frac{1}{\varepsilon^{2/3}})\right) \period
	\end{cases}
\end{align}
Here, the exact solution is given by
\begin{align}
	\varphi_{exact}(x) = \Ai(-\frac{x}{\varepsilon^{2/3}}) + \ii \Bi(-\frac{x}{\varepsilon^{2/3}}) \Comma \nonumber
\end{align}
where $\Ai$ and $\Bi$ denote the Airy functions of first and second kind, respectively (e.g., see \cite[Chap.\ 9]{Olver2010NISTHO}, \cite[Sec.\ 5.1]{Krner2021WKBbasedSW}). Here the phase (\ref{phase}) is exactly computable.

Let us first investigate the convergence results from Theorems~\ref{thm_main} and \ref{thm_simpl}. In Figure~\ref{plot:airy_err_vs_h} we plot the global error $\max_{n\leq N}\lVert U(x_{n})-U_{n}\rVert_{\infty}$ as a function of the step size $h$, for several values of $\eps$. The left plot shows the error for WKB3 (solid lines with asterisks) and WKB3s (dashed lines with circles), and the right plot shows the error for WKB2.
\begin{figure}%[H]
	\centering
	\includegraphics[scale=0.75]{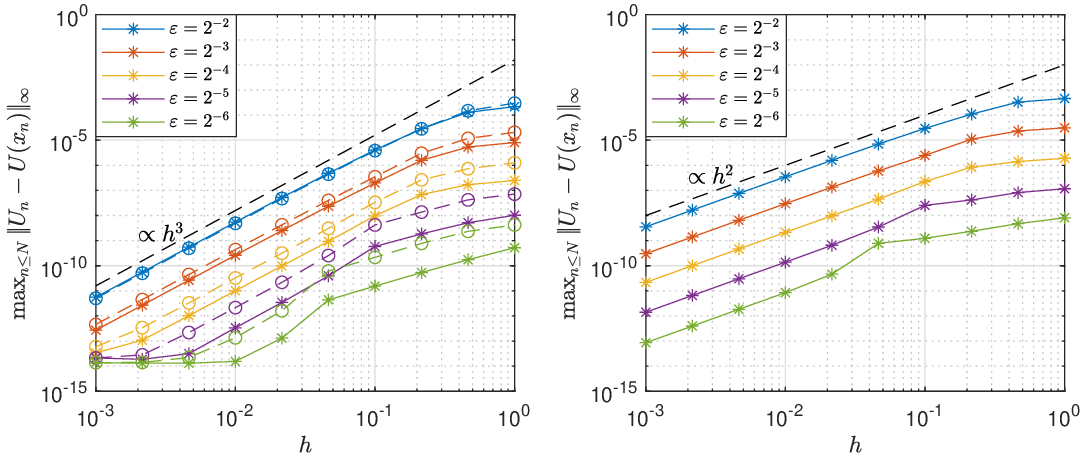}
	\caption{Global errors $\max_{n\leq N}\lVert U(x_{n})-U_{n}\rVert_{\infty}$ for the Airy equation (\ref{eqn:Airy}) as functions of the step size $h$, for several $\eps$-values. Left: WKB3 (solid lines with asterisks) and WKB3s (dashed lines with circles). Right: WKB2.}
	\label{plot:airy_err_vs_h}
\end{figure}
As indicated by the dashed black line in the left plot, the error for WKB3 and WKB3s clearly decays like $\mathcal{O}(h^{3})$, in accordance with the third order estimates (\ref{global_error_U}) and (\ref{global_error_U_simp}). For small step sizes $h$ and $\eps=2^{-5},2^{-6}$ the error saturates at approximately $10^{-14}$, due to rounding errors. Further, for a fixed step size $h$, the error for all three methods decreases with $\eps$, demonstrating the $\eps$-asymptotical correctness of each method. According to estimate (\ref{global_error_U}), we expect the error for WKB3 to behave here like $\mathcal{O}(\eps^{4})$, since no shown error curve decays like $\mathcal{O}(h^{4})$, suggesting that the $\mathcal{O}_{\eps,h}(\max(\eps,h))$-factor in (\ref{global_error_U}) is equal to $\mathcal{O}(\eps)$ for the shown $\eps$-values. Further, estimate (\ref{global_error_U_simp}) suggests that the error for WKB3s decreases like $\mathcal{O}(\eps^{3})$, which is also the expected error behavior of WKB2, according to \cite{Arnold2011WKBBasedSF}. However, Figure~\ref{plot:airy_err_vs_h} reveals that the error for each method shows an $\eps$-order that is $0.3$\,-\,$0.9$ higher than expected theoretically. This is due to the oscillatory behavior of the consistency error of each method, which leads to cancellation effects in successive integration steps. Indeed, this phenomenon was already observed and analyzed in \cite[§3.3]{Arnold2011WKBBasedSF}, and appears in all figures of the present paper. %The reason for WKB3 to show here a higher $\eps$-order than WKB3s and WKB2 is the $\mathcal{O}_{\eps,h}(\max(\eps,h))$-factor in error estimate (\ref{global_error_U}), which is equal to $\mathcal{O}(\eps)$, if $\eps\geq ch$ for some (hidden) constant $c>0$. %Indeed, we can see this behavior even for $h=1$ and $\eps=2^{-6}$, suggesting that for this example it holds $c\leq 2^{-6}$.
%However, we note that, asymptotically, as $\eps\to 0$, the error for WKB3 is still only $\mathcal{O}(\eps^{3})$, see (\ref{global_error_U}).

Figure~\ref{plot:airy_times_vs_err} contains work-precision diagrams, which correspond to the computations from Figure~\ref{plot:airy_err_vs_h}.
\begin{figure}%[H]
	\centering
\includegraphics[scale=0.75]{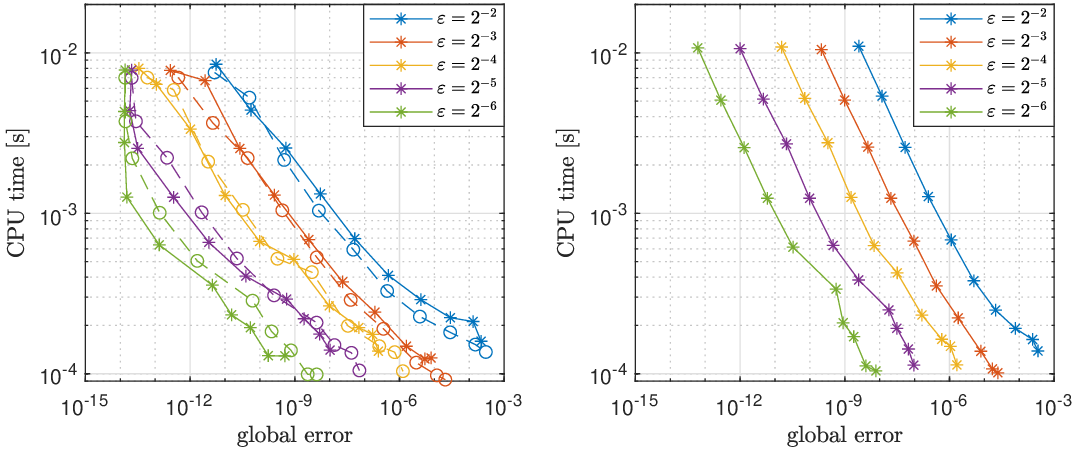}
	\caption{CPU times vs.\ global errors $\max_{n\leq N}\lVert U(x_{n})-U_{n}\rVert_{\infty}$ for the Airy equation (\ref{eqn:Airy}) on the spatial interval $[1, 2]$, for several $\eps$-values. Left: WKB3 (solid lines with asterisks) and WKB3s (dashed lines with circles). Right: WKB2.}
	\label{plot:airy_times_vs_err}
\end{figure}
That is, CPU times (measured in seconds) are plotted against the global errors, resulting from using the different step sizes $h$.
Clearly, in the same amount of time, WKB3 and WKB3s both reach much smaller global errors, when compared to WKB2. For instance, with a CPU time of approximately $10^{-3}$ seconds, the global errors using WKB3 or WKB3s lie between $10^{-14}$\,-\,$10^{-13}$ and $10^{-8}$ for the different $\eps$-values, whereas the errors lie between $10^{-11}$ and $10^{-7}$\,-\,$10^{-6}$ when using WKB2. Vice versa, to reach a fixed accuracy, WKB3 and WKB3s need significantly less CPU time than WKB2. For example, to produce a global error of approximately $10^{-8}$ for $\eps=2^{-2}$, WKB3 and WKB3s need approximately $10^{-3}$ seconds, whereas WKB2 needs around $5\cdot 10^{-3}$ seconds. The time difference is even larger when considering smaller values of $\eps$: E.g., to reach an accuracy of approximately $10^{-13}$ for $\eps=2^{-6}$, WKB3 and WKB3s need around $6\cdot 10^{-4}$ and $10^{-3}$ seconds, respectively. By contrast, the computational time to reach the same accuracy with WKB2 is around $10^{-2}$ seconds. Hence, we conclude from Figure~\ref{plot:airy_times_vs_err} that WKB3 and WKB3s are both more efficient than WKB2 (for this example), with CPU times being smaller up to a factor of ten. The difference between WKB3 and WKB3s is not so obvious: Indeed, for $\eps=2^{-2}$ and $2^{-3}$ we conclude from the left plot of Figure~\ref{plot:airy_times_vs_err} that WKB3s needs less time to reach the same accuracy as WKB3. However, for the smaller values $\eps=2^{-6}$ and $\eps=2^{-5}$, WKB3 becomes more efficient. %E.g., to reach an accuracy of approximately $10^{-13}$ for $\eps=2^{-6}$, WKB3 needs around $6\cdot 10^{-4}$ seconds, whereas WKB3s needs $10^{-3}$ seconds.

In Figures~\ref{plot:airy_err_vs_h_vpa}-\ref{plot:airy_times_vs_err_vpa} analogous plots are shown for much smaller values of $\eps$.
\begin{figure}%[H]
	\centering
	\includegraphics[scale=0.75]{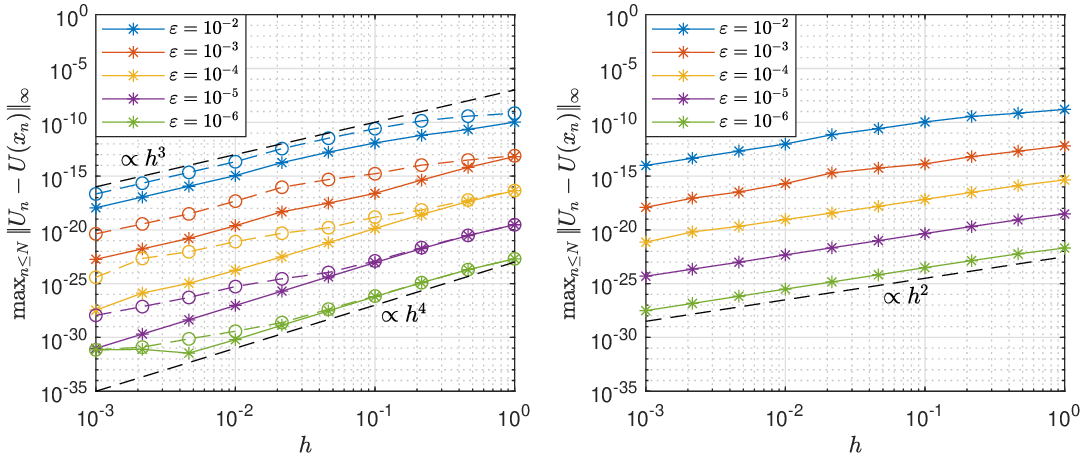}
	\caption{Global errors $\max_{n\leq N}\lVert U(x_{n})-U_{n}\rVert_{\infty}$ for the Airy equation (\ref{eqn:Airy}) on the spatial interval $[1, 2]$ as a function of the step size $h$, for several $\eps$-values. Left: WKB3 (solid lines with asterisks) and WKB3s (dashed lines with circles). Right: WKB2.}
	\label{plot:airy_err_vs_h_vpa}
\end{figure}
According to the convergence plot on the left of Figure~\ref{plot:airy_err_vs_h_vpa}, when $\eps\ll h$, the error for WKB3 even decreases like $\mathcal{O}(h^{4})$, e.g., for $\eps=10^{-4},10^{-5},10^{-6}$, as indicated by the bottom dashed black line. This behavior agrees well with estimate (\ref{global_error_U}), see also Remark \ref{remark_h4}. Indeed, the $\mathcal{O}_{\eps,h}(\max(\eps,h))$-factor from estimate (\ref{global_error_U}) is equal to $\mathcal{O}(h)$ in these cases. Hence, for large step sizes $h$, we cannot expect an $\mathcal{O}(\eps^{4})$ error behavior for WKB3 anymore (as in Figure~\ref{plot:airy_err_vs_h}). Instead, we observe that (e.g., for $h=1$), the error for WKB3, WKB3s, and WKB2 (see the right plot) roughly decreases like $\mathcal{O}(\eps^{3})$. This is in good agreement with estimates (\ref{global_error_U}) and (\ref{global_error_U_simp}).
\begin{figure}%[H]
	\centering
	\includegraphics[scale=0.75]{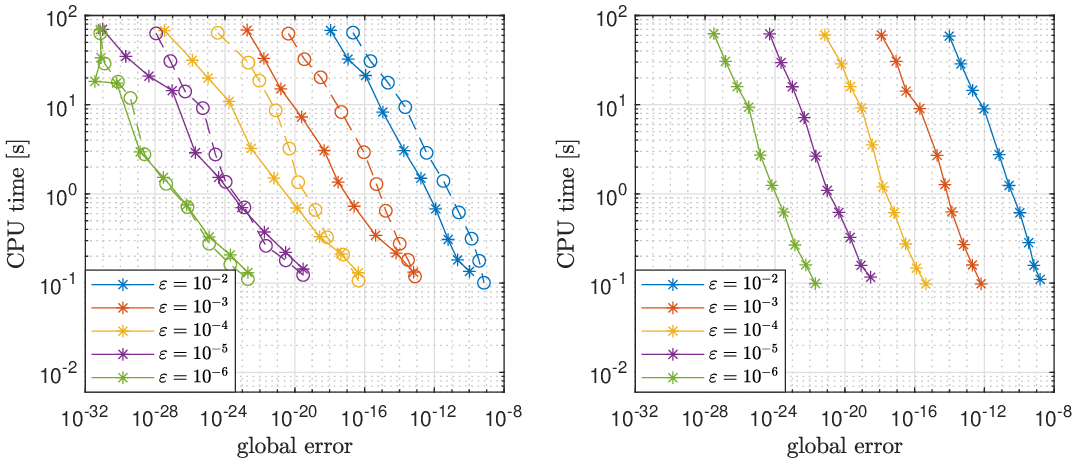}
	\caption{CPU times vs.\ global errors $\max_{n\leq N}\lVert U(x_{n})-U_{n}\rVert_{\infty}$ for the Airy equation (\ref{eqn:Airy}) on the spatial interval $[1, 2]$ as a function of the step size $h$, for several $\eps$-values. Left: WKB3 (solid lines with asterisks) and WKB3s (dashed lines with circles). Right: WKB2.}
	\label{plot:airy_times_vs_err_vpa}
\end{figure}
In the work-precision diagrams of Figure~\ref{plot:airy_times_vs_err_vpa} one can clearly observe that WKB3 and WKB3s are again more efficient than WKB2. E.g., with a CPU time of approximately $3$ seconds, the global errors for the different $\eps$-values lie between $10^{-29}$ and $10^{-13}$\,-\,$10^{-12}$ when using WKB3 and WKB3s, whereas they lie between $10^{-25}$ and $10^{-11}$ when using WKB2. Moreover, we observe also a big difference between WKB3 and WKB3s for $\eps=10^{-2},10^{-3},10^{-4},10^{-5}$: Indeed, for each of these $\eps$-values, and rather small global errors (incurred by small step sizes $h$), WKB3 needs significantly less CPU time than WKB3s. This is due to the $\mathcal{O}(h^{4})$ behavior of the error for WKB3 in these cases. Vice versa, for a fixed accuracy, the difference between the needed CPU times for WKB3, WKB3s, and WKB2 is very large. For instance, for $\eps=10^{-4}$ and a global error of approximately $10^{-21}$, the CPU times for WKB3, WKB3s, and WKB2 are approximately $1.5$ seconds, $8$ seconds, and $60$ seconds, respectively. Hence, for the same accuracy, WKB3 is faster up to a factor of $40$, when compared to WKB2.

Next, let us investigate numerically the estimates (\ref{global_error_U_perturbed}), (\ref{global_error_U_simp_numphase}). To this end, we compute in the following the phase (\ref{phase}) numerically (even though it is exactly computable for this example) and investigate results obtained with WKB3 and WKB3s. In Figure~\ref{plot:airy_err_vs_h_numphase} we show on the left the global error of WKB3 and WKB3s when using the composite Simpson rule to compute the approximate phase $\tilde{\phi}$. Note that this implies $\lVert \phi-\tilde{\phi}\rVert_{L^{\infty}(I)}=\mathcal{O}(h^{4})$. Indeed, for $\eps= 2^{-4},2^{-5},2^{-6}$ and $h\geq 10^{-2}$ the error curves for WKB3 and WKB3s behave like the first error term in (\ref{global_error_U_perturbed}) and (\ref{global_error_U_simp_numphase}), respectively, i.e.\ like $\mathcal{O}_{\eps,h}(h^{4}/\eps)$ due to Simpson's rule, as indicated by the lower dashed black line. For $h=1$ and $\eps\leq 2^{-3}$ we even observe an inversion of all error curves for WKB3 and WKB3s, i.e., the error increases with $\eps$. However, for small step sizes and large $\eps$ (e.g., $h\leq 10^{-1}$ and $\eps=2^{-3}$) the second error term in the estimates (\ref{global_error_U_perturbed}), (\ref{global_error_U_simp_numphase}) becomes dominant. Indeed, as indicated by the upper dashed black line, the error curves for both methods are clearly third order. Further, in this case the inversion of the error curves w.r.t.\ $\eps$ disappears. Moreover, for small values of $\eps$ and $h$ (e.g., $\eps\leq 2^{-4}$ and $h=10^{-3}$) the error curves reach a saturation level at approximately $10^{-14}$\,-\,$10^{-13}$, due to rounding errors.
\begin{figure}%[H]
	\centering
	\includegraphics[scale=0.75]{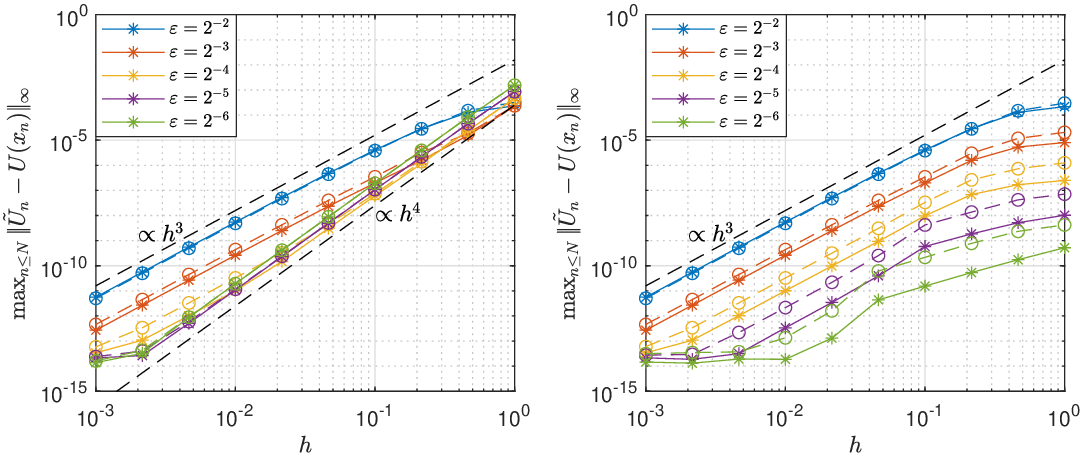}
	\caption{Global errors $\max_{n\leq N}\lVert U(x_{n})-\tilde{U}_{n}\rVert_{\infty}$ for the Airy equation (\ref{eqn:Airy}) on the spatial interval $[1, 2]$ as a function of the step size $h$, for several $\eps$-values. The solid lines with asterisks correspond to WKB3, and the dashed lines with circles correspond to WKB3s. Here, we used an approximate phase $\tilde{\phi}$ computed with two different methods: Left: Composite Simpson's rule. Right: Clenshaw-Curtis algorithm based on $N_{cheb}=17$ Chebyshev grid points, with barycentric interpolation.}
	\label{plot:airy_err_vs_h_numphase}
\end{figure}

For the right plot of Figure~\ref{plot:airy_err_vs_h_numphase} we computed the approximate phase $\tilde{\phi}$ with a spectral method, as already used in \cite{Arnold2022WKBmethodFT}. That is, we use the well-known Clenshaw-Curtis algorithm \cite{Clenshaw1960AMF} to approximate the phase (\ref{phase}) on a Chebyshev grid for the whole interval $[1,2]$ with $N_{cheb}=17$ points, and then use barycentric interpolation (e.g., see \cite{Berrut2004BarycentricLI}) to obtain the approximate phase $\tilde{\phi}$ at the uniform grid which is used for the WKB schemes. We note that, using this method, the phase is approximated to machine precision. Indeed, we observe in the right plot of Figure~\ref{plot:airy_err_vs_h_numphase} that the first error term in the estimates (\ref{global_error_U_perturbed}), (\ref{global_error_U_simp_numphase}) is numerically invisible, as the entire plot is almost indistinguishable from the left plot of Figure~\ref{plot:airy_err_vs_h}, where the exact phase was used.
\anton{
\subsection{Comparison with results from program {\tt riccati}}\label{sec:comparison}
Next we shall compare the performance of our two schemes from Section~\ref{chap:scheme} with the program {\tt riccati}, which is based on the Riccati defect correction method from \cite{Agocs2024}. We used its {\tt Python} implementation, version 1.1.3 \cite{Agocs2023}.\footnote{\anton{downloaded from {\tt https://github.com/fruzsinaagocs/riccati}}} While the WKB schemes from Section~\ref{chap:scheme} solve \eqref{schroedinger_IVP} on a user-defined grid, {\tt riccati} is a sophisticated adaptive solver that couples a Riccati phase function solution in oscillatory regions with a spectral method for smooth regions. {\tt riccati} first solves the ODE on a course grid that is adaptively chosen, depending on a user defined error tolerance. In a second phase called \emph{dense output} \cite{Agocs2020a}, the numerical solution can then be evaluated by interpolation on a desired grid. Given these conceptual differences between the two schemes, an ``objective'' efficiency comparison is delicate. For our tests we made the following set up choices: To be able to compare the errors on the same grid, we used the \emph{dense output} option of {\tt riccati} to evaluate its solution a posteriori on the same grid as used in the WKB method. Moreover, we use three different tolerances {\tt eps} for {\tt riccati} in the following tests. In contrast to Section~\ref{sec:Ex1}, here we present the error of the wave function $\varphi$, as this is the output of {\tt riccati} 1.1.3.}

\anton{As a first test we consider again the IVP \eqref{eqn:Airy} on the interval $[1,2]$. Thus, the work-precision diagram for WKB3 and WKB3s in Figure \ref{plot:airy_vs_riccati2} is an extension of the left plot in Figure \ref{plot:airy_times_vs_err}, but now we include four additional (small) values of $\varepsilon$. Like for Figure \ref{plot:airy_times_vs_err} the results for WKB3,  WKB3s are produced by prescribing different step sizes $h\in[10^{-3},1]$. As the two methods considered here are structurally very different, their efficiency results in Figure \ref{plot:airy_vs_riccati2} look very different. While {\tt riccati} always needed a CPU time of at least $10^{-3}$s, even for a low tolerance of $10^{-6}$ (not shown here), coarse discretizations within WKB3 often led to times between $10^{-4}$s and $10^{-3}$s. For {\tt riccati}, the CPU times appear to be quite independent of the output density in this example. For the WKB3 method, the error decays with $\varepsilon\to0$ initially, i.e.\ for (fixed) coarse meshes and hence small execution times. However, for (fixed) fine meshes (and hence small errors), this ordering is reversed in both WKB3 and {\tt riccati}: The error grows with $\varepsilon\to0$. This is apparently due to saturation in the roundoff errors and the bad conditioning of evaluating highly oscillatory functions, see the discussion in \S4 of \cite{Agocs2024}.\footnote{\anton{Since the \MATLAB implementation of the Airy function, appearing here as the exact reference solution, is quite inaccurate for large arguments, we replace it there by its series representation, see Remark 5.2 in \cite{Krner2021WKBbasedSW} for details.}}
To sum up, let us compare the results of WKB3 with those of {\tt riccati} with tolerance {\tt eps} set to $10^{-13}$: For $\varepsilon$ small, i.e.\ $\varepsilon=2^{-10}, 2^{-9}, 2^{-8}$, WKB3 can reach the same (and sometimes even higher) accuracy as {\tt riccati}, but a bit faster. Moreover, WKB3s tends to be slightly faster than WKB3. For $\varepsilon$ larger, {\tt riccati} is faster. 
\begin{figure}%[H]
	\centering
	\includegraphics[scale=0.75]{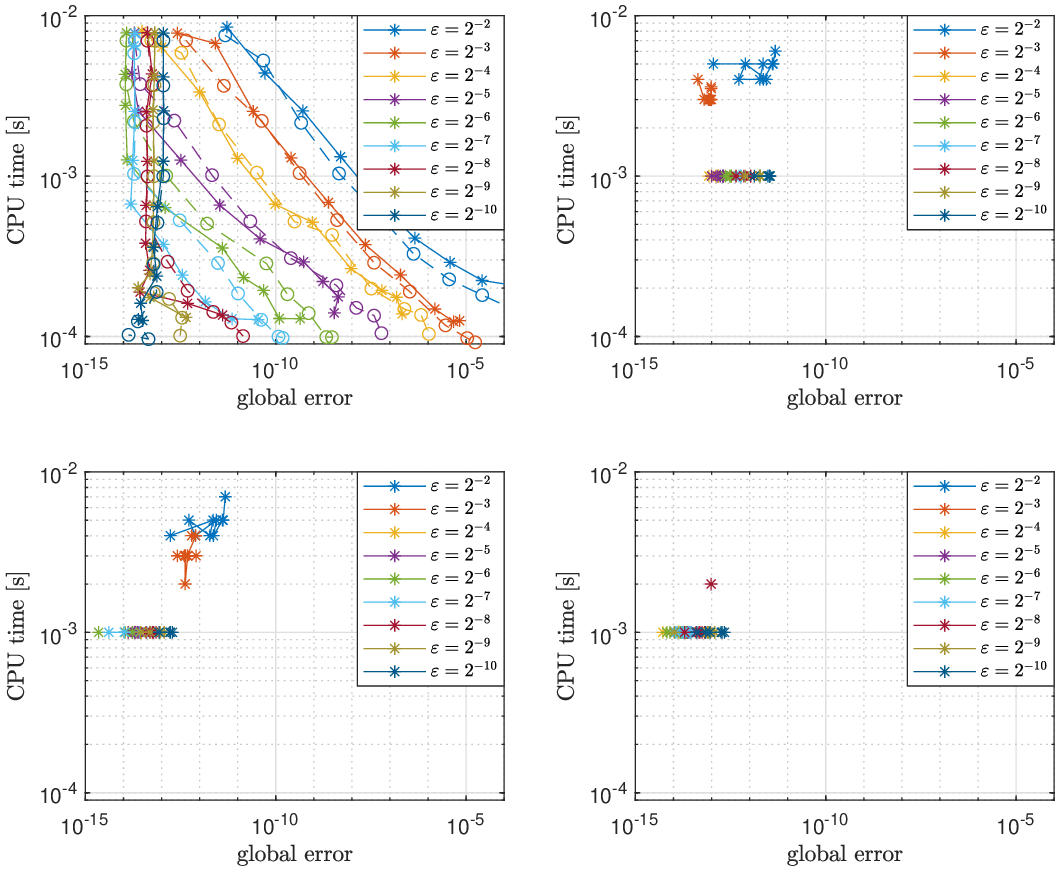}
	\caption{\anton{CPU times vs.\ global errors $\max_{n\leq N}\lVert \varphi(x_{n})-\varphi_{n}\rVert_{\infty}$ for the Airy equation (\ref{eqn:Airy}) on the spatial interval $[1, 2]$, for several $\eps$-values. Top left: WKB3 (solid lines with asterisks) and WKB3s (dashed lines with circles). Top right, bottom left and right: {\tt riccati} with $tol=10^{-11}$, $tol=10^{-12}$, $tol=10^{-13}$.}}
	\label{plot:airy_vs_riccati2}
\end{figure}
}

\anton{
As a second test we consider the IVP from \eqref{eqn:Airy} on the interval $[1,100]$. While the work-precision diagram for WKB3, WKB3s behaves qualitatively as for the interval $[1,2]$, see Figure \ref{plot:airy_vs_riccati100}, the dense output of {\tt riccati} now dominates the CPU time for fine grids. This explains the vertical lines in the second to forth panel of Figure \ref{plot:airy_vs_riccati100}. 
As before, let us compare the results of WKB3 with those of {\tt riccati} with tolerance {\tt eps} set to $10^{-13}$: For the smallest value of $\varepsilon$, i.e. $2^{-8}$, the accuracy obtained with WKB3 on a course grid is slightly better and the CPU time is slightly shorter than with {\tt riccati}. For $\varepsilon=2^{-7}$ the results of both methods are almost the same when using the second coarsest grid for WKB3. Again WKB3s tends to be a bit faster than WKB3, but often is a bit less accurate. For larger values of $\varepsilon$, {\tt riccati} is clearly faster than WKB3, while also yielding smaller errors}. 
%\jannis{(Wir sollten hier vielleicht noch darauf aufmerksam machen, dass der Fehler bei riccati für fallendes $\varepsilon$ steigt. Das ist für mich ein \textbf{essentieller} Unterschied in dieser Figure. Oder nicht? Außerdem sollten wir diesen kurzen ``Vergleich'' der beiden Methoden eventuell noch mit einem kurzen Schlusssatz bzw. Absatz abschließen, der nochmal klar macht, dass dies hier natürlich kein umfangreicher Vergleich ist, sondern hier nur indikativ gezeigt wird, dass die bessere Effizienz einer der beiden Methoden situationsbedingt ist. Oder wie siehst du das?)}

\anton{To conclude, the efficiency advantages of the two compared methods depend very much on the concrete situation, but for larger values of $\varepsilon$, {\tt riccati} is clearly more efficient.}

\begin{figure}%[H]
	\centering
	\includegraphics[scale=0.75]{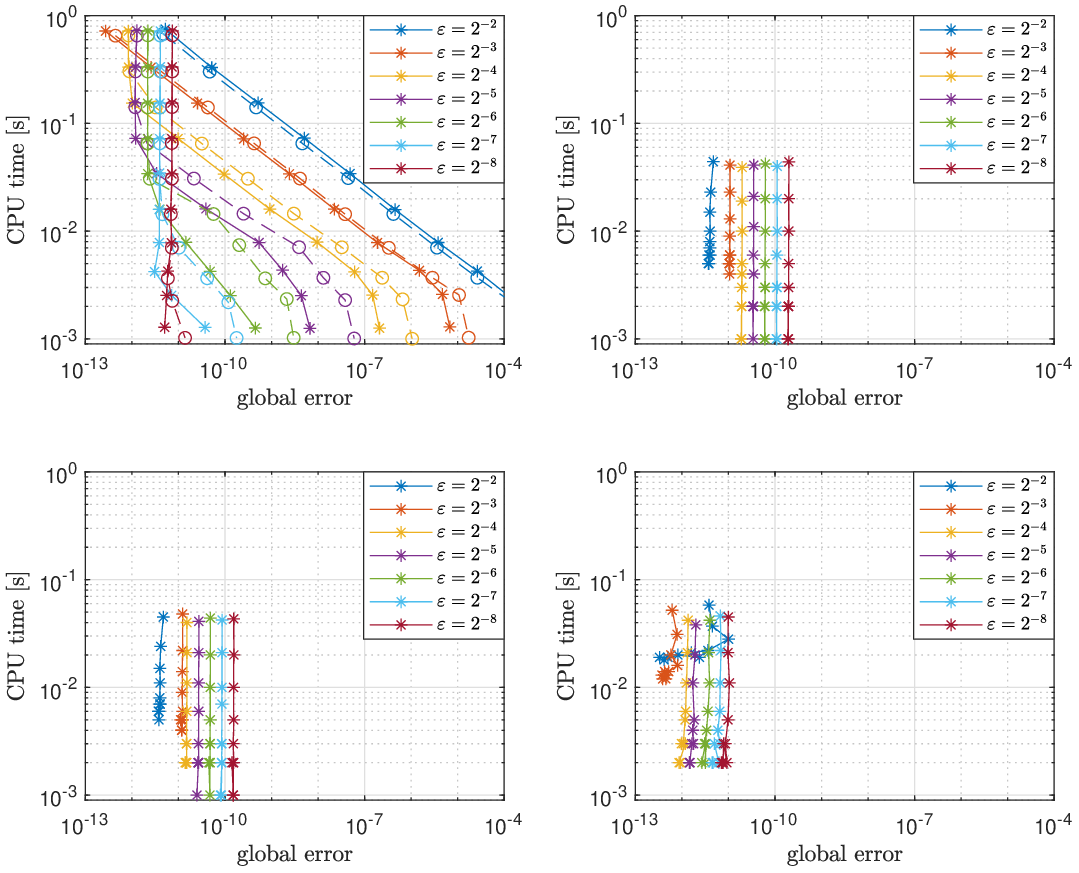}
	\caption{\anton{CPU times vs.\ global errors $\max_{n\leq N}\lVert \varphi(x_{n})-\varphi_{n}\rVert_{\infty}$ for the Airy equation (\ref{eqn:Airy}) on the spatial interval $[1, 100]$, for several $\eps$-values. Top left: WKB3 (solid lines with asterisks) and WKB3s (dashed lines with circles). Top right, bottom left and right: {\tt riccati} with $tol=10^{-11}$, $tol=10^{-12}$, $tol=10^{-13}$.}}
	\label{plot:airy_vs_riccati100}
\end{figure}
%
%
%

%%%%%%%%%%%%%%%%%%%%%%%%%%%%%%%%%%%%%%%%%%%%%%%%
\subsection{Second example}
As a second example let us consider the following IVP:
\begin{align}\label{eqn:expx}
	\begin{cases}
		\varepsilon^{2}\varphi^{\prime\prime  }(x) +\ee^{x} \varphi(x) = 0 \Comma \quad x \in [0, 1] \Comma \\
		\varphi(0) = 1 \Comma \\
		\varepsilon\varphi^{\prime}(0) = 0 \period
	\end{cases}
\end{align}
Here, the exact solution is given by\footnote{We computed the exact solution by using the Symbolic Math Toolbox of \MATLAB.}
\begin{align}
    \varphi_{exact}(x)=    
    \frac{J_{0}\left(\frac{2}{\eps}\ee^{x/2}\right)Y_{1}\left(\frac{2\sqrt{\ee}}{\eps}\right) - Y_{0}\left(\frac{2}{\eps}\ee^{x/2}\right)J_{1}\left(\frac{2\sqrt{\ee}}{\eps}\right)}{Y_{1}\left(\frac{2\sqrt{\ee}}{\eps}\right)J_{0}\left(\frac{2}{\eps}\right)-J_{1}\left(\frac{2\sqrt{\ee}}{\eps}\right)Y_{0}\left(\frac{2}{\eps}\right)}\Comma
\end{align}
where $J_{\nu}$ and $Y_{\nu}$ denote the Bessel functions of first and second kind of order $\nu$, respectively (e.g., see \cite[Chap.\ 10]{Olver2010NISTHO}). Again the phase (\ref{phase}) is exactly computable here.

In Figure~\ref{plot:expx_glob_error_U} we plot again the global error $\max_{n\leq N}\lVert U(x_{n})-U_{n}\rVert_{\infty}$ as a function of the step size $h$, for several $\eps$-values. The left plot shows the results for WKB3 and WKB3s, and the right plot shows the results for WKB2. The dashed black line in the left plot confirms the third order estimates (\ref{global_error_U}), (\ref{global_error_U_simp}), as the error for WKB3 and WKB3s clearly decays like $\mathcal{O}(h^{3})$. For small values of $h$ and small values of $\eps$ (e.g.\ $h\leq 10^{-2}$ and $\eps=2^{-5}$) the error curves include rounding errors, due to the used double precision arithmetic. Moreover, when using the step size $h=1$, the error for WKB3 is $\mathcal{O}(\eps^{4})$, whereas the error for WKB3s is $\mathcal{O}(\eps^{3})$. This is again due to the $\mathcal{O}_{\eps,h}(\max(\eps,h))$-factor in estimate (\ref{global_error_U}), which for the shown $\eps$-values yields an $\mathcal{O}(\eps)$-factor.
\begin{figure}%[H]
	\centering
	\includegraphics[scale=0.75]{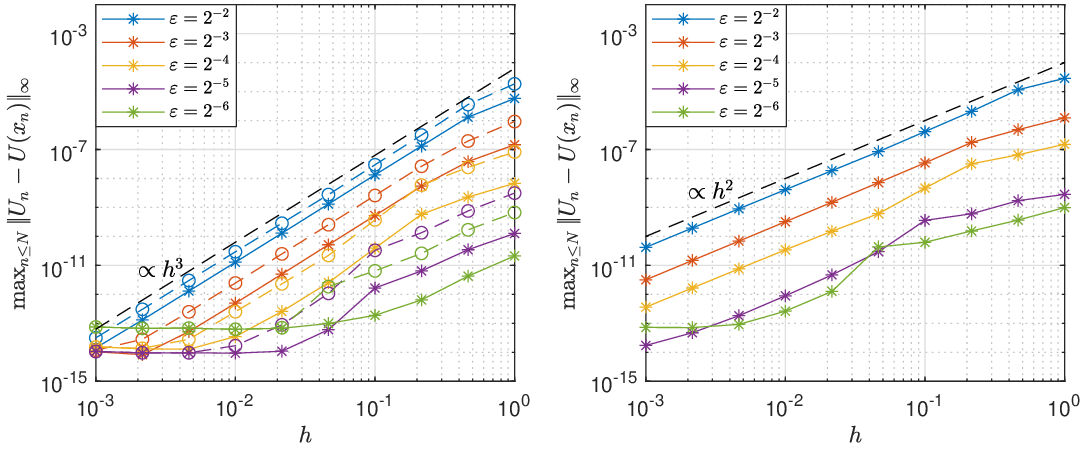}
	\caption{Global errors $\max_{n\leq N}\lVert U(x_{n})-U_{n}\rVert_{\infty}$ for the IVP (\ref{eqn:expx}) as functions of the step size $h$, for several $\eps$-values. Left: WKB3 (solid lines with asterisks) and WKB3s (dashed lines with circles). Right: WKB2.}
	\label{plot:expx_glob_error_U}
\end{figure}

Figure~\ref{plot:expx_glob_error_U_vs_time} shows work-precision diagrams, which correspond to the computations from Figure~\ref{plot:expx_glob_error_U}. We observe that, in the same amount of time, WKB3 and WKB3s both yield much smaller global errors than WKB2. E.g., with a CPU time of approximately $10^{-3}$\,-\,$1.5\cdot 10^{-3}$ seconds, the global errors for WKB3 and WKB3s are between $10^{-14}$ and $5\cdot 10^{-11}$ for the different $\eps$-values. In contrast, the errors for WKB2 lie between $10^{-13}$ and $5\cdot 10^{-9}$. Vice versa, in order to attain a fixed accuracy, WKB3 and WKB3s need significantly less CPU time, when compared to WKB2. This can be observed, e.g., for $\eps=2^{-2}$ and an accuracy of approximately $10^{-10}$: Then, WKB3 and WKB3s both need less than $10^{-3}$ seconds, whereas WKB2 needs approximately $5\cdot 10^{-2}$ seconds. This time difference is even larger for smaller values of $\eps$. We conclude from Figure~\ref{plot:expx_glob_error_U_vs_time} that WKB3 and WKB3s are more efficient than WKB2 (in the present example). Similar to the previous example, the difference between WKB3 and WKB3s is not so obvious here. Indeed, we observe on the left of Figure~\ref{plot:expx_glob_error_U_vs_time} that the blue curves for $\eps=2^{-2}$ are almost identical. But for $\eps=2^{-4}$ and a global error of approximately $3\cdot 10^{-12}$ the CPU time for WKB3 is $4\cdot 10^{-4}$, whereas the CPU time for WKB3s is $6\cdot 10^{-4}$. Overall, we conclude that the efficiency difference between WKB3 and WKB3s is small for this example. However, for a given accuracy and small $\eps$-values, WKB3 tends to be faster than WKB3s.
\begin{figure}%[H]
	\centering
	\includegraphics[scale=0.75]{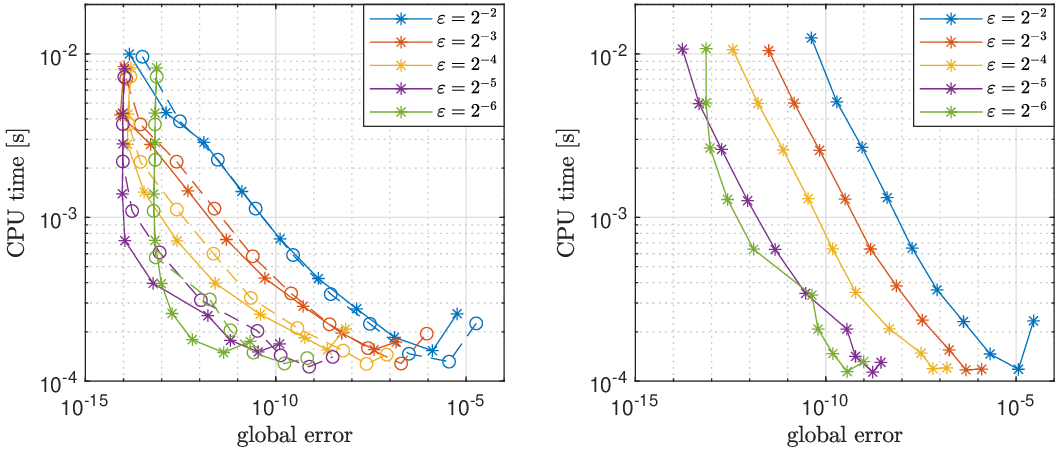}
	\caption{CPU times vs.\ global errors $\max_{n\leq N}\lVert U(x_{n})-U_{n}\rVert_{\infty}$ for the IVP (\ref{eqn:expx}), for several $\eps$-values. Left: WKB3 (solid lines with asterisks) and WKB3s (dashed lines with circles). Right: WKB2.}
	\label{plot:expx_glob_error_U_vs_time}
\end{figure}

Next, we compute for the schemes WKB3 and WKB3s the phase (\ref{phase}) numerically in two ways and investigate the estimates (\ref{global_error_U_perturbed}), (\ref{global_error_U_simp_numphase}). For this, we shall again apply the composite Simpson rule as well as the Clenshaw-Curtis algorithm along with barycentric interpolation. In Figure~\ref{plot:expx_glob_error_U_numphase} we show on the left the global error $\max_{n\leq N}\lVert U(x_{n})-\tilde{U}_{n}\rVert_{\infty}$ for WKB3 and WKB3s as a function of the step size when using Simpson's rule. We observe that for both methods the first error term in (\ref{global_error_U_perturbed}), (\ref{global_error_U_simp_numphase}), i.e. the $\mathcal{O}_{\eps,h}(h^{4}/\eps)$-term, dominates for $\eps\leq 2^{-3}$ and all used step sizes $h$. Indeed, this is indicated by the dashed black line which decays like $\mathcal{O}(h^{4})$. For $\eps=2^{-2}$, however, this error term merely dominates for step sizes $h\geq 10^{-1}$. Indeed, when using smaller step sizes, i.e. $h\leq 10^{-1}$, the error for $\eps=2^{-2}$ behaves like $\mathcal{O}(h^{3})$ for both methods. Note also that for $h=1$ we can observe again an inversion of the shown error curves w.r.t.\ $\eps$, when compared to the left plot of Figure~\ref{plot:expx_glob_error_U}.

On the right of Figure~\ref{plot:expx_glob_error_U_numphase} we show the errors for WKB3 and WKB3s when the phase is obtained with the spectral method. As indicated by the dashed black line, the first error term in (\ref{global_error_U_perturbed}), (\ref{global_error_U_simp_numphase}) is essentially eliminated. The error curves for both methods are clearly third order and the entire plot is very similar to the left plot of Figure~\ref{plot:expx_glob_error_U}.
\begin{figure}%[H]
	\centering
	\includegraphics[scale=0.75]{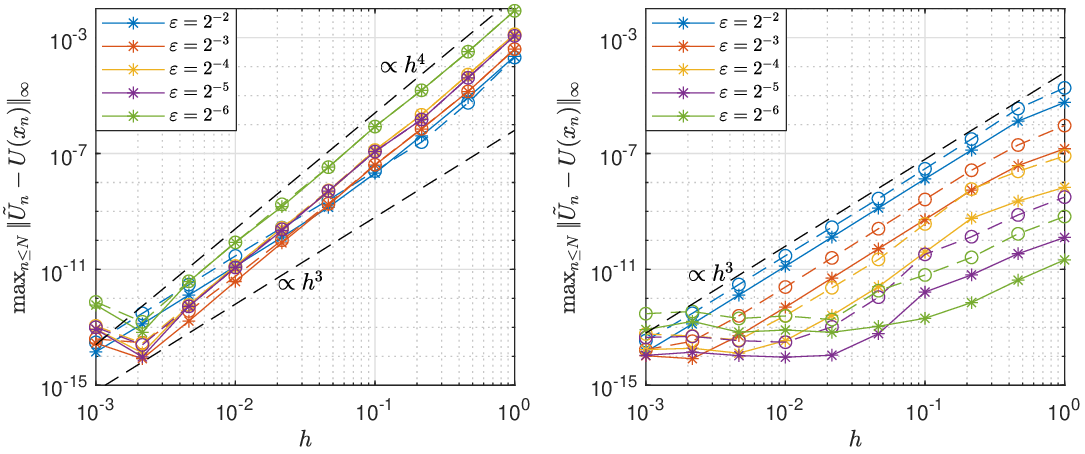}
	\caption{Global errors $\max_{n\leq N}\lVert U(x_{n})-\tilde{U}_{n}\rVert_{\infty}$ for the IVP (\ref{eqn:expx}) as a function of the step size $h$, for several $\eps$-values. The solid lines with asterisks correspond to WKB3, and the dashed lines with circles correspond to WKB3s. Here, we used an approximate phase $\tilde{\phi}$ computed with two different methods: Left: Composite Simpson's rule. Right: Clenshaw-Curtis algorithm based on $N_{cheb}=17$ Chebyshev grid points, with barycentric interpolation.}
	\label{plot:expx_glob_error_U_numphase}
\end{figure}
\section{Conclusion}\label{chap:conclusion}
In this paper we developed a third order one-step method to efficiently compute solutions to the highly oscillatory 1D Schrödinger equation (\ref{schroedinger}), \anton{and we derived theoretically sound error estimates for it.} The method is based on the WKB-transformation from \cite{Arnold2011WKBBasedSF}, which was already used there for the development of a first and second order scheme. Like those two \anton{previous} methods, the presented method has the property of asymptotical correctness w.r.t.\ the small parameter $\eps$ in case of an explicitly computable phase. Additionally, in scenarios where $\eps\ll h$ (with $h$ being the grid size), the error decays \anton{even like $\mathcal O(h^4)$.}

\anton{In efficiency comparisons with the program {\tt riccati} \cite{Agocs2024, Agocs2023} we found that the WKB3-marching method has slight advantages w.r.t.\ accuracy and/or runtime for very small values of $\varepsilon$, while {\tt riccati} is clearly more efficient for larger $\varepsilon$.}

\anton{In this paper, the step size $h$ in WKB3 was treated as an input parameter. The extension to an adaptive step size control, as in \cite{Krner2021WKBbasedSW, Agocs2020}, and dense output (e.g. based on linear interpolation of the $Z$-function from \eqref{Zsystem} or interpolation/extrapolation of the $\varphi(x_n)$ with the WKB-function $\varphi_2^{WKB}$ from \eqref{2ndorderwkb}) will be left for future work. }

\section*{Declaration of competing interest}
The authors declare that they have no known competing financial interests or personal relationships that could have appeared to influence the work reported in this paper.
\section*{Acknowledgements}
Both authors acknowledge support from the Austrian Science Fund (FWF) project \href{https://doi.org/10.55776/F65}{10.55776/F65}, and the author A. Arnold also by the bi-national FWF-project I3538-N32. We are grateful to the numerous suggestions by the anonymous referees that helped to improve this manuscript.

\bibliographystyle{abbrv}
\bibliography{references}

\end{document}